\documentclass[11pt]{amsart}
\usepackage{latexsym}
\usepackage{amsmath,amsthm,amssymb, amscd,color}
\usepackage{mathrsfs}
\usepackage[all]{xy}
\usepackage[all]{xy}
\usepackage{tikz}
\usepackage{tikz-cd}
\usepackage{adjustbox}
\usetikzlibrary{snakes}
\usepackage[utf8]{inputenc}
\usepackage{enumerate}
\usepackage{breqn}
\usepackage{hyperref}

\usepackage{blkarray} 
\usepackage{framed}

\numberwithin{equation}{section}

\theoremstyle{plain}
\newtheorem{theorem}{Theorem}[section]
\newtheorem{lemma}[theorem]{Lemma}
\newtheorem{corollary}[theorem]{Corollary}
\numberwithin{equation}{section}

\theoremstyle{definition}
\newtheorem{definition}[theorem]{Definition}
\newtheorem{example}[theorem]{Example}

\newtheorem{proposition}[theorem]{Proposition}
\newtheorem{remark}[theorem]{Remark}

\theoremstyle{remark}

\newcommand{\RR}{\mathbb{R}}

\newcommand{\NN}{\mathbb{N}}

\begin{document}

\title[Quadratic, homogeneous and Kolmogorov vector fields]{Quadratic, homogeneous and Kolmogorov vector fields on $S^1\times S^2$ and $S^2 \times S^1$}

\author[S. Jana]{Supriyo Jana}
\address{Department of Mathematics, Indian Institute of Technology Madras, India}
\email{supriyojanawb@gmail.com}

\author[S. Sarkar]{Soumen Sarkar}
\address{Department of Mathematics, Indian Institute of Technology Madras, India}
\email{soumen@iitm.ac.in}

\subjclass[2020]{34A34, 34C14, 34C40, 37J35}

\keywords{Polynomial vector fields, First integrals, Algebraic hypersurfaces, Quadratic vector fields, Kolmogorov systems, Lotka-Volterra systems}

\date{\today}
\dedicatory{}

\abstract In this paper, we consider the following two algebraic hypersurfaces
$$S^1\times S^2=\{(x_1,x_2,x_3,x_4)\in \RR^4:(x_1^2+x_2^2-a^2)^2 + x_3^2 + x_4^2 -1=0;\, a>1\}$$ and $$S^2\times S^1=\{(x_1,x_2,x_3,x_4)\in \RR^4:(x_1^2+x_2^2+x_3^2-b^2)^2+x_4^2-1=0;\, b>1\}$$embedded in $\RR^4$. We study polynomial vector fields in $\RR^4$ separately, having $S^1\times S^2$ and $S^2\times S^1$ invariant by their flows. We characterize all linear, quadratic, cubic Kolmogorov and homogeneous vector fields on $S^1\times S^2$ and $S^2\times S^1$. We construct some first integrals of these vector fields and find which of the vector fields are Hamiltonian. We give upper bounds for the number of the invariant meridian and parallel hyperplanes of these vector fields. In addition, we have shown that the upper bounds are sharp in many cases. 
\endabstract

\maketitle


\section{Introduction}
Let $\mathbb{R}[x_1,x_2,x_3,x_4]$ be the polynomial ring over real in 4 variables. Consider the following polynomial differential system in $\mathbb{R}^4$
\begin{equation}\label{system}
    \frac{dx_i}{dt}=P_i(x_1,x_2,x_3,x_4)\,; i=1,...,4
\end{equation}
  and its associated polynomial vector field 
\begin{equation}\label{vector-field}
     \chi = \sum\limits_{i=1}^4 P_i \frac{\partial}{\partial x_i}
\end{equation}
where $P_1,..., P_4\in \mathbb{R}[x_1,x_2,x_3,x_4]$ of degrees $n_1,..., n_4$ respectively. For simplicity, we may write $\chi:=(P_1,...,P_4)$. The number $n:=\max\{n_1,..., n_4\}$ is called the degree of the polynomial vector field \eqref{vector-field}. Polynomial vector fields of degrees 1, 2 and 3 are called linear, quadratic and cubic vector fields, respectively.
\par Polynomial vector fields on various algebraic manifolds embedded in either of $\RR^2$ and $\RR^3$ are already well-studied, for instance, see \cite{LlRe13, BoLlVa13, LlMe11, LlPe06, LlZh11, ChLlPaWa19}. However, only very few investigations are there in $\RR^4$ (\cite{LlXi17, LlMu21}). The advantage of having a polynomial vector field on an algebraic manifold is that if a solution curve of a differential system has an intersection with the manifold, then the whole solution curve lies on that manifold.

Kolmogorov system in $\RR^4_{+}$ has been studied in \cite{LlXi17}. Then the integrability of a class of Lotka-Volterra and Kolmogorov systems in $\RR^n$ have been explored in \cite{LlRaRa20}. We note that a Lotka-Volterra system (i.e., a degree two polynomial Kolmogorov system) in $\RR^N$ describes the evolution of $N$ conflicting species in population biology. It appears in many different areas like neural networks, laser physics, plasma physics, etc.( see for instance \cite{ChHsWu98, HoSi98,Lo20}).

A first integral of the associated vector field of a differential system reduces the dimension of the system by one, which makes the analysis of the system easier. Therefore, finding first integrals is extremely important in the qualitative theory of differential equations. However, the study of the existence or non-existence of first integrals is a difficult problem in general.If the vector field is Hamiltonian, then one easily gets at least one first integral. In addition, if there is another independent first integral of a Hamiltonian system then it is called integrabile in Liouville sense. A Liouville integrable system satisfy the assumptions
of the famous Liouville theorem, see \cite{Ar19}. So, the corresponding system can be solved by quadratures, i.e., by solving a finite number of algebraic equations and computing
a finite number of definite integrals. Also,  Darboux theory of integrability \cite[Chapter 8]{DuLlAr06} is very useful to compute first integrals. 

We consider two algebraic hypersurfaces embedded in $\RR^4$, the products of spheres $S^1\times S^2$ and $S^2\times S^1$, which are manifolds of dimension 3. In \cite{BeSa22}, Benny and the second author proved that product of two arbitrary spheres can be expressed as the zero set of a polynomial. From \cite{BeSa22}, we found
\begin{equation}
\label{eq:s12}
S^1\times S^2=\{(x_1,x_2,x_3,x_4)\in \RR^4:(x_1^2+x_2^2-a^2)^2 + x_3^2 + x_4^2 -1=0;~a>1\}
\end{equation}
and
\begin{equation}
\label{eq:s21}
S^2\times S^1=\{(x_1,x_2,x_3,x_4)\in \RR^4:(x_1^2+x_2^2+x_3^2-b^2)^2+x_4^2-1=0;~b>1\}.
\end{equation}

In this paper, we characterize all linear, quadratic, cubic Kolmogorov and homogeneous vector fields on $S^1\times S^2$ and $S^2\times S^1$. We construct some first integrals of these vector fields and find which of the vector fields are Hamiltonian. We give upper bounds for the number of the invariant meridian and parallel hyperplanes of these vector fields.

 We note that there is a diffeomorphism, possibly not polynomial, between $S^1 \times S^2$ and $S^2 \times S^1$. However, the algebraic equations in \eqref{eq:s12} and \eqref{eq:s21} are algebraically different.  Therefore, we study them as different hypersurfaces in $\RR^4$. 


This paper is organized in the following way. In Section \ref{sec:prel}, we present some basic definitions and useful results on polynomial vector fields.
We recall the criteria to search for the invariant hypersurfaces of a vector field. 

In Section \ref{sec:pvf_on_s1s2}, we characterize linear vector fields on $S^1\times S^2$.
We give a necessary and sufficient condition when a quadratic and a cubic Kolmogorov vector field in $\RR^4$ is a vector field  on $S^1 \times S^2$, see Theorem \ref{thm:quad_fv_s1s2} and \ref{thm:cubic_kolmo_s1s2} respectively. We show that there is no Lotka-Volterra vector field on $S^1 \times S^2$, see Corollary \ref{cor:lv-s1s2}. We prove that there is no Hamiltonian of a quadratic and cubic Kolmogorov vector field on $S^1 \times S^2$, see Theorem \ref{quadratic-hamiltonian} and \ref{thm:cubic-ham-s12} respectively. Moreover, we constructed rational first integral for them in those theorems. If $\chi = (P_1, P_2, P_3, P_4)$ is a vector field in $\RR^4$ with each $P_i$ is a non-zero homogeneous polynomial, then we give some sufficient as well as some necessary conditions for ${(x_1^2+x_2^2-a^2)^2 + x_3^2 + x_4^2 -1}$ to be a first integral of $\chi$, see Lemma \ref{deg-m-n-s12} and Theorem \ref{thm:degree-relation} respectively. We characterize Type-$n$ vector fields on $S^1\times S^2$, see Theorem \ref{type-n-s12} and study some properties of these vector fields.


In Section \ref{sec:pvf_s2s1}, we characterize linear vector fields on $S^2\times S^1$. We exhibit a necessary and sufficient condition when a quadratic and a cubic Kolmogorov vector field in $\RR^4$ is a vector field  on $S^2 \times S^1$, see Theorem \ref{quadratic-thm-s21} and \ref{thm:kolmo_vf_s2s1} respectively. We prove that there is no Lotka-Volterra vector field on $S^1 \times S^2$; see Corollary \ref{cor_lvvf_s2s1}. We show that there is no Hamiltonian of a quadratic and cubic Kolmogorov vector field on $S^2 \times S^1$, see Theorem \ref{thm:noHam_quad} and Theorem \ref{thm:noHam_kolm} respectively. Moreover, we constructed rational first integral for them in those theorems. If $\chi = (P_1, P_2, P_3, P_4)$ is a vector field in $\RR^4$ with each $P_i$ is a homogeneous polynomial, then we discuss some sufficient as well as some necessary conditions for ${(x_1^2+x_2^2 +x_3^2-a^2)^2 + x_4^2 -1}$ is a first integral of $\chi$, see Lemma \ref{deg-m-n-s21} and Lemma \ref{lem:P_4-0} respectively. We characterize Pseudo Type-$n$ vector field on $S^2 \times S^1$, see Theorem \ref{pseudo-type-n} and study some properties of these vector fields.

In Section \ref{sec:mer_par}, we investigate the maximum number of invariant meridian and parallel hyperplanes of the vector fields on $S^1 \times S^2$ and $S^2 \times S^1$. We show when these bounds can be reached. We conclude this paper by establishing that there is no degree one polynomial diffeomorphism between $S^1 \times S^2$ and $S^2 \times S^1$, see Theorem \ref{thm:nopolydiff}.

\section{Preliminaries}\label{sec:prel}
Let $f \in \RR[x_1,...,x_4]$ be a non-constant polynomial. The set $\{f = 0\} \subset \RR^4$ is called an invariant algebraic hypersurface of the vector field $\chi$ if there exists a polynomial ${K\in \RR[x_1,...,x_4]}$ such that
    $$\chi f=\sum\limits_{i=1}^4 P_i\frac{\partial f}{\partial x_i}=Kf.$$
Here, the polynomial $K$ is called the cofactor of $\chi$ for hypersurface $\{f = 0\}$.

One can observe that the gradient $(\frac{\partial f}{\partial x_1},\cdots,\frac{\partial f}{\partial x_4})$ of
$f$ is orthogonal to the vector field $\chi = (P_1,...,P_4)$ at the points of the algebraic hypersurface $\{f = 0\}$. So the vector field $\chi$ is tangent to the hypersurface $\{f = 0\}$. Hence the hypersurface $\{f = 0\}$ is formed by the orbits of the vector field $\chi$. This justifies the name ``invariant algebraic hypersurface'' given to $\{f = 0\}$ because it is invariant under the
flow determined by $\chi$. Therefore, if a solution curve of system \eqref{system} has a point on the algebraic hypersurface $f$,
then the whole solution curve is contained in $f$.

In this paper, we consider polynomial vector fields such that $S^1\times S^2$ and $S^2\times S^1$ of \eqref{eq:s12} and \eqref{eq:s21} respectively are invariant algebraic hypersurfaces of $\chi$. If $S^1\times S^2$ is invariant under $\chi$, then $\chi$ is called a \textit{polynomial vector field on $S^1\times S^2$}. Similarly, we call $\chi$ a \textit{polynomial vector field on $S^2\times S^1$} if $S^2\times S^1$ is invariant under $\chi$. 

Such vector fields are called \textit{polynomial vector fields on $S^1\times S^2$} and \textit{polynomial vector fields on $S^2\times S^1$} respectively.

One of the best tools in order to search for invariant algebraic hypersurfaces is the following. Let
$W$ be a vector subspace of $\RR[x_1,...,x_4]$ generated by the independent
polynomials $v_1 ,... , v_l$, i.e., $W=\langle v_1,...,v_{\ell} \rangle$. The extactic polynomial of $\chi$ associated with $W$ is the polynomial
$$\mathcal{E}_W(\chi)= 
\begin{vmatrix}
v_1 & \dots &v_{\ell}\\
\chi (v_1) &\dots &\chi (v_{\ell})\\
\dots &\dots &\dots\\
\chi^{{\ell}-1}(v_1) &\dots &\chi^{{\ell}-1} (v_{\ell})
\end{vmatrix}
$$
where $\chi^j (v_i) = \chi^{j-1}( \chi (v_i))$ for all $i, j$. From the properties of determinant, it follows that the definition of the extactic polynomial is independent of the chosen
basis of $W$.
The proof of the following is similar to the proof
of \cite[Proposition 1]{LlMe07}.
\begin{proposition}\label{extactic-polynomial}
Let $\chi$ be a polynomial vector field in $\mathbb{R}^n$ and $W$
a finite dimensional vector subspace of $\mathbb{R}[x_1, x_2, \cdots , x_n]$ with $\dim(W) > 1$. Then, an invariant algebraic hypersurface given by $\{f = 0\}$ for the vector field $\chi$, with $f\in W$, is a factor of $\mathcal{E}_W(\chi)$.
\end{proposition}
\par The multiplicity of an invariant algebraic hypersurface $f = 0$ with $f\in W$ is the largest positive integer
$k$ such that $f^k$ divides the polynomial $\mathcal{E}_W(\chi)$ when $\mathcal{E}_W(\chi)\neq 0$, otherwise the multiplicity
is infinite. For more details on the multiplicity, see \cite{ChLlPe07, LlZh09}.
\begin{definition}
Let $U$ be an open subset of $\RR^4$. A non-constant analytic map $H \colon U \to \mathbb{R}$ is called a first integral
of the vector field \eqref{vector-field} on $U$ if $H$ is constant on all solution curves of system \eqref{system} contained in $U$ ; i.e. $H(x_1(t),x_2(t), x_3(t),x_4(t)) =$ constant for all values of $t$ for which the solution $(x_1(t),x_2(t), x_3(t),x_4(t))$ is defined and contained in $U$.
\end{definition}
Note that $H$ is
a first integral of the vector field \eqref{vector-field} on $U$ if and only if $\chi H=0$
on $U$.

\begin{definition}
Given a non-constant analytic function $H \colon \RR^4\to \RR$, the vector field
$$\chi=-\frac{\partial H}{\partial x_2}\frac{\partial}{\partial x_1}+\frac{\partial H}{\partial x_1}\frac{\partial}{\partial x_2}-\frac{\partial H}{\partial x_4}\frac{\partial}{\partial x_3}+\frac{\partial H}{\partial x_3}\frac{\partial}{\partial x_4}$$
is called a Hamiltonian vector field in $\RR^4$ with Hamiltonian $H$.
\end{definition}
So, if the system \eqref{system} is written by
$$P_1=-\frac{\partial H}{\partial x_2}, P_2=\frac{\partial H}{\partial x_1}, P_3=-\frac{\partial H}{\partial x_4}, P_4=\frac{\partial H}{\partial x_3}$$
then the vector field \eqref{vector-field} is Hamiltonian. Observe that a Hamiltonian $H$ of a Hamiltonian vector field is a first integral of that vector field.
\begin{definition}
A Hamiltonian system in $\RR^4$ with Hamiltonian $H$  is called integrable in the Liouville sense in $\RR^4$ if the system has a first integral $F$, which is independent with the Hamiltonian $H$.
\end{definition}
Note that the Poisson bracket between
$H$ and $F$ is zero if and only if $F$ is a first integral of the Hamiltonian system in $\RR^4$ with Hamiltonian $H$. See details in \cite{Ar19}.

\begin{definition} The system \eqref{system} is called a polynomial Kolmogorov system in $\RR^4$ when $P_i=x_i\psi_i$ and its associated vector field is called polynomial Kolmogorov vector field where $\psi_i\in \RR[x_1,...,x_4]$ for $1\leq i\leq 4$.

In addition, if $\psi_i$ is a linear polynomial for each $i$, then the polynomial Kolmogorov system is called the Lotka-Volterra system in $\RR^4$, and its associated vector field is called Lotka-Volterra vector field.
\end{definition}

\begin{definition}
   Suppose $\chi=(P_1,...,P_4)$ is a polynomial vector field in $\RR^4$. We say $\chi$ is \textit{Type-$n$ vector field} when $P_1,...,P_4$ are homogeneous polynomials of degree $n$ and $\chi$ is \textit{Pseudo Type-$n$ vector field} when $P_1,P_2,P_3$ are homogeneous polynomials of degree $n$ but $P_4= 0$.
\end{definition}

\section{Polynomial vector fields on $S^1\times S^2$}\label{sec:pvf_on_s1s2}
In this section, we study and characterize polynomial vector fields on $S^1\times S^2$. Recall the presentation of $S^1\times S^2\subset \RR^4$ from \eqref{eq:s12}. A vector field $\chi=(P_1,...,P_4)$ in $\RR^4$ is a vector field on $S^1\times S^2$ if
\begin{equation}\label{vectorfield-s12}
    4(x_1^2+x_2^2-a^2)(P_1x_1 + P_2x_2)+2(P_3x_3 + P_4x_4)=K((x_1^2+x_2^2-a^2)^2 + x_3^2 + x_4^2 -1)
\end{equation}
for some $K\in \RR[x_1,...,x_4]$.
\begin{lemma}\label{polynomial-ch-s12}
    Let $Q_1,Q_2,R_1,R_2\in \RR[x_1,...,x_4]$ be such that  ${Q_1R_1+Q_2R_2}$ is the zero polynomial and $\gcd(R_1,R_2)=1$. Then $Q_1=AR_2,Q_2=-AR_1$ for some polynomial ${A \in \RR[x_1,...,x_4].}$
\end{lemma}
\begin{proof}
    $Q_1R_1+Q_2R_2=0$ implies that $Q_1R_1=-Q_2R_2.$ Hence $R_2$ divides $Q_1$ and $R_1$ divides $Q_2$ since $\gcd(R_1,R_2)=1.$ Assuming $Q_1=AR_2$ and $Q_2=A^{'}R_1$, we get $(A+A^{'})R_1R_2=0$. This implies that $A^{'}=-A$.
\end{proof}
\begin{theorem}\label{thm:quad_fv_s1s2}
    Let $\chi=(P_1,...,P_4)$ be a quadratic vector field in $\RR^4$. Then $\chi$ is a quadratic vector field on $S^1\times S^2$ if and only if
    \begin{equation}\label{quadratic-ch-s12}
        \begin{split}
            P_1&=\frac{1}{4}Kx_1+fx_2,\\
            P_2&=\frac{1}{4}Kx_2-fx_1,\\
            P_3&=\frac{k_3}{2}(-a^2(x_1^2+x_2^2)+x_3^2+x_4^2+ a^4-1)+g x_4,\quad \text{and}\\
            P_4&=\frac{k_4}{2}(-a^2(x_1^2+x_2^2)+x_3^2+x_4^2 + a^4-1)-g x_3
        \end{split}
    \end{equation}
    where $K=k_3x_3+k_4x_4$ for some $k_3, k_4 \in \RR$ and $f, g$ are linear polynomials. Moreover, $K$ is the cofactor of $\chi$ for $S^1\times S^2$.
\end{theorem}
\begin{proof}
Suppose $\chi$ is a quadratic vector field on $S^1\times S^2$. So \eqref{vectorfield-s12} gives the following. 
    \begin{dmath}\label{quadratic-s12}
        4(x_1^2+x_2^2-a^2)(\sum_{i=1}^2 P_ix_i) + 2(\sum_{i=3}^4 P_ix_i) = (\sum\limits_{j=1}^4 k_jx_j)((x_1^2+x_2^2-a^2)^2 + x_3^2 + x_4^2 -1),
    \end{dmath} since there is no constant term on the left side.    We write 
    \begin{equation}\label{homogeneous-division}
    P_i=P_i^{(2)}+P_i^{(1)}+P_i^{(0)}
    \end{equation}
    where $P_i^{(j)}$ is the degree $j$ homogeneous part of the polynomial $P_i$. Comparing the degree 5 terms in \eqref{quadratic-s12}, we get
    \begin{dmath*}
        4(x_1^2+x_2^2)(\sum_{i=1}^2 P_i^{(2)}x_i)=(\sum\limits_{j=1}^4 k_jx_j)(x_1^2+x_2^2)^2.
    \end{dmath*}
    Hence
    \begin{dmath}\label{deg-5-s12}
        \sum_{i=1}^2 P_i^{(2)}x_i=\frac{1}{4}(\sum\limits_{j=1}^4 k_jx_j)(x_1^2+x_2^2).
    \end{dmath}
    Also comparing the degree $4$ terms in \eqref{quadratic-s12}, we get
    \begin{dmath*}
        4(x_1^2+x_2^2)(\sum_{i=1}^2 P_i^{(1)}x_i)=0.
    \end{dmath*}
    Hence $\sum\limits_{i=1}^2 P_i^{(1)}x_i=0$. Then, by Lemma \ref{polynomial-ch-s12}, we obtain $P_1^{(1)}=Ax_2$ and $P_2^{(1)}=-Ax_1$ for some $A\in \RR$.
    
    In particular, if $x_3=x_4=0$, then \eqref{quadratic-s12} gives the following.
    \begin{dmath}\label{x_3_x_4-0-s12}
         4(x_1^2+x_2^2-a^2)(\sum_{i=1}^2 P_i(x_1,x_2,0,0)x_i) =(k_1x_1+k_2x_2)((x_1^2+x_2^2-a^2)^2 -1).
    \end{dmath}
    Comparing the degree 3 terms in \eqref{x_3_x_4-0-s12}, we get
\begin{dmath*}
    4(x_1^2+x_2^2)(\sum_{i=1}^2 P_i^{(0)}(x_1,x_2,0,0)x_i )-4a^2( \sum_{i=1}^2 P_i^{(2)}(x_1,x_2,0,0)x_i )=-2a^2(\sum_{i=1}^2 k_ix_i)(x_1^2+x_2^2).
\end{dmath*}
Then, for $x_3=0=x_4$, the above equation together with \eqref{deg-5-s12} implies the following.
\begin{dmath}\label{x_3-0-x_4-s12-quad}
    4(\sum_{i=1}^2 P_i^{(0)}x_i)=-a^2(k_1x_1+k_2x_2).
\end{dmath}
Also comparing the degree 1 terms in \eqref{x_3_x_4-0-s12}, we get
\begin{dmath}
\label{x_3-0-x_4-s12-linear}
    -4a^2(\sum_{i=1}^2 P_i^{(0)}x_i)=(a^4-1)(k_1x_1+k_2x_2).
\end{dmath}
Now, from \eqref{x_3-0-x_4-s12-quad} and \eqref{x_3-0-x_4-s12-linear}, we obtain $k_1=k_2=0$. Hence $P_1^{(0)}=P_2^{(0)}=0$. Therefore, from \eqref{deg-5-s12}, we get the following. 
\begin{dmath}\label{homogeneous-2-s12}
\sum_{i=1}^2 P_i^{(2)}x_i=\frac{1}{4}(k_3x_3+k_4x_4)(x_1^2+x_2^2).
\end{dmath}
Taking $K:=k_3x_3+k_4x_4$, the above equation implies that 
$$(P_1^{(2)}-\frac{1}{4}Kx_1)x_1+(P_2^{(2)}-\frac{1}{4}Kx_2)x_2=0.$$
By Lemma \ref{polynomial-ch-s12},
$$P_1^{(2)}=\frac{1}{4}Kx_1+Bx_2 \text{ and } P_2^{(2)}=\frac{1}{4}Kx_2-Bx_1$$ where $B$ is a linear homogeneous polynomial.
Therefore, we can get $P_1$ and $P_2$ since all the homogeneous parts are obtained for both the polynomials.

Let us now find $P_3$ and $P_4$.
Comparing the degree 1 terms in \eqref{quadratic-s12}, we get
$$P_3^{(0)}=\frac{a^4 -1}{2}k_3, P_4^{(0)}=\frac{a^4 -1}{2}k_4.$$
Comparing the degree 2 terms in \eqref{quadratic-s12}, we get $P_3^{(1)}x_3+P_4^{(1)}x_4=0$ which follows that $P_3^{(1)}=Cx_4, P_4^{(1)}=-Cx_3$ for some $C\in \RR$ using Lemma \ref{polynomial-ch-s12}. Now comparing the degree 3 terms in \eqref{quadratic-s12}, we get
\begin{dmath*}
    -4a^2(\sum_{i=1}^2 P_i^{(2)}x_i)+2\sum_{i=3}^4 P_i^{(2)}x_i=(k_3x_3+k_4x_4)(-2a^2(x_1^2+x_2^2)+x_3^2+x_4^2).
\end{dmath*}
Then using \eqref{homogeneous-2-s12} in the above equation, we have
\begin{dmath*}
    \sum_{i=3}^4 P_i^{(2)}x_i=\frac{1}{2}(k_3x_3+k_4x_4)(-a^2(x_1^2+x_2^2)+x_3^2+x_4^2).
\end{dmath*}
This implies
$$(P_3^{(2)}-\frac{1}{2}k_3(-a^2(x_1^2+x_2^2)+x_3^2+x_4^2))x_3+(P_4^{(2)}-\frac{1}{2}k_4(-a^2(x_1^2+x_2^2)+x_3^2+x_4^2))x_4=0.$$
Therefore, by Lemma \ref{polynomial-ch-s12}, 
$$P_3^{(2)}=\frac{1}{2}k_3(-a^2(x_1^2+x_2^2)+x_3^2+x_4^2)+Dx_4 \quad \text{and} \quad P_4^{(2)}=\frac{1}{2}k_4(-a^2(x_1^2+x_2^2)+x_3^2+x_4^2)-Dx_3$$
where $D\in \RR[x_1,...,x_4]$ is a linear homogeneous polynomial. Hence, $P_3, P_4$ are completely determined.

 If $P_1, \ldots, P_4$ are given by \eqref{quadratic-ch-s12}, then they satisfy \eqref{vectorfield-s12}. Thus, the converse part is true.
\end{proof}
\begin{corollary}\label{cor:linear-s12}
$\chi=(P_1,...,P_4)$ is a linear vector field on $S^1\times S^2$ if and only if there exists $\alpha,\beta\in \RR$ such that 
$$P_1=\alpha x_2, P_2=-\alpha x_1, P_3=\beta x_4, P_4=-\beta x_3.$$

    Observe that if $P_1,...,P_4\neq 0$ then a linear vector field $\chi=(P_1,...,P_4)$ on $S^1\times S^2$ is a Type-1 vector field.
\end{corollary}
\begin{proof}
    The vector field \eqref{quadratic-ch-s12} on $S^1\times S^2$ is linear if and only if $K=0$ and $f,g$ are constants. Hence the result follows.
\end{proof}
\begin{corollary}\label{cor:lv-s1s2}
There is no Lotka-Volterra vector field on $S^1\times S^2$.
\end{corollary}
\begin{proof}
Suppose $\chi=(P_1,...,P_4)$ defines a Lotka-Volterra vector field on $S^1\times S^2$. Then $x_i$ divides $P_i$ for  $i\in \{1,...,4\}$. Since $\chi$ is a quadratic vector field, it will be of the form \eqref{quadratic-ch-s12}. Now $x_1$ divides $P_1$ and $x_2$ divides $P_2$ implies $x_1$ divides $f$ and $x_2$ divides $f$ respectively. So $x_1x_2$ divides $f$. Therefore, $f$ must be zero since it is linear.

Again, $x_3$ divides $P_3$ implies $k_3=0$ and $x_3$ divides $g$. Similarly, $x_4$ divides $P_4$ implies $k_4=0$ and $x_4$ divides $g$. So $x_3x_4$ divides $g$. Therefore, $g$ must be zero since it is linear. Hence $P_1=\cdots=P_4=0$.
\end{proof}

\begin{theorem}\label{quadratic-hamiltonian}
There is no Hamiltonian of a quadratic vector field on $S^1\times S^2$. However, every quadratic vector field on $S^1\times S^2$ has a rational first integral.
\end{theorem}
\begin{proof}
Any quadratic vector field on $S^1\times S^2$ can be described by \eqref{quadratic-ch-s12}.  
    Suppose $H$ is a Hamiltonian of the quadratic vector field \eqref{quadratic-ch-s12} on $S^1\times S^2$, where $f=f_0+\sum\limits_{i=1}^4 f_ix_i$ and ${g=g_0+\sum\limits_{i=1}^4 g_ix_i}$; $f_i, g_i \in \mathbb{R}$. Since ${-\frac{\partial H}{\partial x_2}=P_1 = \frac{1}{4}K x_1 + (f_0 + \sum\limits_{i=1}^4 f_i x_i) x_2}$, we get $$H=-\frac{1}{4}Kx_1x_2-\frac{1}{2}(f_0+f_1x_1+f_3x_3+f_4x_4)x_2^2-\frac{1}{3}f_2x_2^3+H_1(x_1,x_3,x_4),$$ where $H_1$ is a function independent of $x_2$. Now, from $\frac{\partial H}{\partial x_1}=P_2,$ we get $$-\frac{1}{4} K x_2 -\frac{1}{2} f_1 x_2^2 + \frac{\partial H_1}{\partial x_1} = \frac{1}{4} K x_2 -f_0x_1 - (\sum_{i=1}^4 f_ix_i) x_1.$$ 
    
    Thus, $(-\frac{1}{2}K + f_1x_2 + f_2x_1)x_2 = \frac{\partial H_1}{\partial x_1} + (f_0 + f_3x_3 + f_4x_4)x_1$. Therefore,
    \begin{equation}\label{eq:hamiltonian-s12-1}
        K=f_1=f_2=0 ~\mbox{and}~ \frac{\partial H_1}{\partial x_1}=-(f_0+f_3x_3+f_4x_4)x_1,
    \end{equation}
since $H_1$ is independent of $x_2$ and $K = k_3x_3+k_4x_4$.
Hence, $$H_1=-\frac{1}{2}(f_0+f_3x_3+f_4x_4)x_1^2 +H_2(x_3,x_4),$$ where $H_2$ is a function of $x_3,x_4$ only. Therefore, $H=-\frac{1}{2}(f_0+f_3x_3+f_4x_4)(x_1^2+x_2^2)+H_2$. Note that $P_3=gx_4,P_4=-gx_3$ since $K=0$. Now, $-\frac{\partial H}{\partial x_4}=P_3$ implies that $$\frac{1}{2}f_4(x_1^2+ x_2^2) - \frac{\partial H_2}{\partial x_4} = (g_0 + \sum_{i=1}^4 g_i x_i)x_4.$$ Thus,
$$\frac{1}{2}f_4(x_1^2+ x_2^2)-(g_1x_1+g_2x_2)x_4=\frac{\partial H_2}{\partial x_4}+(g_0+g_3x_3+g_4x_4)x_4$$
Therefore,
    \begin{equation}\label{eq:hamiltonian-s12-2}
        f_4=g_1=g_2=0 ~ \quad \mbox{and } \quad \frac{\partial H_2}{\partial x_4}=-(g_0+g_3x_3+g_4x_4)x_4,
    \end{equation}
    since $H_2$ is a funcion of $x_3,x_4$ only. Hence, $$H_2=-\frac{1}{2}(g_0+g_3x_3)x_4^2-\frac{1}{3}g_4x_4^3+H_3(x_3),$$ where $H_3$ is a function of $x_3$ only. Hence $$H =-\frac{1}{2}(f_0+f_3x_3)(x_1^2+x_2^2)-\frac{1}{2}(g_0+g_3x_3)x_4^2-\frac{1}{3}g_4x_4^3+H_3.$$ Again, $\frac{\partial H}{\partial x_3}=P_4$ implies that $-\frac{1}{2}f_3(x_1^2+x_2^2)-\frac{1}{2}g_3x_4^2+H^{'}_3=-(g_0+g_3x_3+g_4x_4).$ Hence \begin{equation}\label{eq:hamiltonian-s12-3}
        f_3=g_3=g_4=0,
    \end{equation}
    since $H_3$ is a function of $x_3$ only. Combining \eqref{eq:hamiltonian-s12-1}, \eqref{eq:hamiltonian-s12-2} and \eqref{eq:hamiltonian-s12-3}, we obtain $f_i=g_i=0$ for $i\in \{1,...,4\}$. Hence the polynomial vector field becomes linear, which contradicts the assumption. So, there is no Hamiltonian of a quadratic vector field on $S^1 \times S^2$.

 Notice that the hypersurfaces given by $(x_1^2+x_2^2-a^2)^2+x_3^2+x_4^2-1=0$ and $x_1^2+x_2^2=0$ are two invariant algebraic hypersurfaces of the vector field \eqref{quadratic-ch-s12} with cofactor $K$ and $\frac{1}{2}K$ respectively. Therefore, by Darboux Integrability Theory (Theorem 5, \cite{LlZh02}), the function ${((x_1^2+x_2^2-a^2)^2+x_3^2+x_4^2-1)(x_1^2+x_2^2)^{-2}}$ is a rational first integral of the vector field \eqref{quadratic-ch-s12}.
\end{proof}
\begin{theorem}\label{thm:cubic_kolmo_s1s2}
Let $\chi=(P_1,...,P_4)$ be a cubic Kolmogorov vector field in $\RR^4$. Then $\chi$ is a vector field on $S^1\times S^2$ if and only if
    \begin{equation}\label{kolmogorov-form-s12}
        \begin{split}
     P_1&=x_1(\frac{1}{4}K +\alpha x_2^2),\\
     P_2&=x_2(\frac{1}{4}K -\alpha x_1^2),\\
     P_3&=x_3(\frac{k_{33}}{2}(-a^2(x_1^2+x_2^2)+x_3^2+x_4^2+a^4-1)+\beta x_4^2),\quad \text{and}\\
     P_4&=x_4(\frac{k_{44}}{2}(-a^2(x_1^2+x_2^2)+x_3^2+x_4^2+a^4-1)-\beta x_3^2)
        \end{split}
    \end{equation}
    where $K=k_{33}x_3^2+k_{44}x_4^2$ with $k_{33}, k_{44} \in \RR$. Moreover, $K$ is the cofactor of $\chi$ for $S^1\times S^2$.
\end{theorem}
\begin{proof}
Consider a cubic Kolmogorov vector field $\chi=(P_1,...,P_4)$ in $\RR^4$ with $$P_i=x_i \psi_i$$ where $\psi_i\in \RR[x_1,...,x_4]$ is a quadratic polynomial for $i=1,...,4$. 
 We write
    $$\psi_i=\psi_i^{(2)}+\psi_i^{(1)}+\psi_i^{(0)}$$ as in \eqref{homogeneous-division}. \par Suppose $\chi$ is a vector field on $S^1\times S^2$. Then it must satisfy
    \begin{dmath}\label{kolmogorov3-s12}
        4(\sum_{i=1}^2 x_i^2-a^2)(\sum_{i=1}^2 \psi_ix_i^2)+2(\sum_{i=3}^4\psi_ix_i^2)=(\sum\limits_{1\leq i\leq j \leq 4} k_{ij}x_ix_j)((x_1^2+x_2^2-a^2)^2 + x_3^2 + x_4^2 -1),
    \end{dmath}
    since there is no constant and degree 1 term on the left side.
    Comparing the degree 6 terms in \eqref{kolmogorov3-s12}, we get
    \begin{dmath*}
        4(x_1^2+x_2^2)(\sum_{i=1}^2 \psi_i^{(2)}x_i^2)=(\sum\limits_{1\leq i\leq j \leq 4} k_{ij}x_ix_j)(x_1^2+x_2^2)^2.
    \end{dmath*}
    Hence
    \begin{dmath}\label{kolmogorov3-deg-6-s12}
        \sum_{i=1}^2 \psi_i^{(2)}x_i^2=\frac{1}{4}(\sum\limits_{1\leq i\leq j \leq 4} k_{ij}x_ix_j)(x_1^2+x_2^2).
    \end{dmath}
    Also comparing the degree $5$ terms in \eqref{kolmogorov3-s12}, we obtain
    \begin{dmath*}
        4(x_1^2+x_2^2)(\sum_{i=1}^2 \psi_i^{(1)}x_i^2)=0.
    \end{dmath*}
    Hence $\sum\limits_{i=1}^2 \psi_i^{(1)}x_i^2=0$. This implies $x_1^2$ divides $\psi_2^{(1)}$ and $x_2^2$ divides $\psi_1^{(1)}$. Therefore, $\psi_1^{(1)},\psi_2^{(2)}$ must be zero since they are linear polynomials.
    
    In particular, if $x_3=x_4=0$, then \eqref{kolmogorov3-s12} gives the following.
    \begin{dmath}\label{kolmogorov3_x_3_x_4-0-s12}
         4(\sum_{i=1}^2 x_i^2-a^2)(\sum_{i=1}^2 \psi_i(x_1,x_2,0,0)x_i^2)=(\sum\limits_{ 1\leq i\leq j \leq 2} k_{ij}x_ix_j)((x_1^2+x_2^2-a^2)^2 -1).
    \end{dmath}
    Comparing the degree 4 terms in \eqref{kolmogorov3_x_3_x_4-0-s12}, we get
 $$   4(\sum_{i=1}^2 x_i^2)(\sum_{i=1}^2 \psi_i^{(0)}(x_1,x_2,0,0)x_i^2)-4a^2(\sum_{i=1}^2 \psi_i^{(2)}(x_1,x_2,0,0)x_i^2)=-2a^2(\sum\limits_{1 \leq i\leq j \leq} k_{ij}x_ix_j)(\sum_{i=1}^2 x_i^2). $$
Then, for $x_3=0=x_4$, the above equation together with \eqref{kolmogorov3-deg-6-s12} implies the following.
\begin{dmath}\label{eq.x_3=0=x_4}
    4(\sum_{i=1}^2 \psi_i^{(0)}x_i^2)=-a^2(\sum\limits_{1 \leq i\leq j \leq 2} k_{ij}x_ix_j).
\end{dmath}
Also, comparing the degree 2 terms in \eqref{kolmogorov3_x_3_x_4-0-s12}, we get
\begin{dmath}\label{eq.deg2}
    -4a^2(\sum_{i=1}^2 \psi_i^{(0)}x_i^2)=(a^4-1)(\sum\limits_{1 \leq i\leq j \leq 2} k_{ij}x_ix_j).
\end{dmath}
From \eqref{eq.x_3=0=x_4} and \eqref{eq.deg2}, we obtain $k_{ij}=0$ for $1\leq i\leq j\leq 2$. Hence $\psi_1^{(0)}=\psi_2^{(0)}=0$. Therefore, from \eqref{kolmogorov3-deg-6-s12}, we get
\begin{dmath}\label{homogeneous-2-kolmogorov3-s12}
\sum_{i=1}^2 \psi_i^{(2)}x_i^2=\frac{1}{4}(\sum\limits_{\substack{1\leq i\leq 4,\, 3 \leq j\leq 4\\ i\leq j}} k_{ij}x_ix_j)(x_1^2+x_2^2).
\end{dmath}
Taking $K :=\sum\limits_{\substack{1\leq i\leq 4,\, 3 \leq j\leq 4\\ i\leq j}} k_{ij}x_ix_j$, the equation \eqref{homogeneous-2-kolmogorov3-s12} implies that
\begin{dmath*}
(\psi_1^{(2)}-\frac{1}{4}K)x_1^2 + (\psi_2^{(2)}-\frac{1}{4})x_2^2=0.
\end{dmath*}
Hence, by Lemma \ref{polynomial-ch-s12}, $\psi_1^{(2)}=\frac{1}{4}K+\alpha x_2^2$ and $\psi_2^{(2)}=\frac{1}{4}K-\alpha x_1^2$ for some $\alpha\in \RR$. So, we have found all the homogeneous components of $\psi_1$ and $\psi_2$.

Let us now find $\psi_3$ and $\psi_4$. Comparing the coefficients of the degree 2 monomials in \eqref{kolmogorov3-s12}, we obtain
$$\psi_3^{(0)}=\frac{a^4 -1}{2}k_{33}, \quad \psi_4^{(0)}=\frac{a^4 -1}{2}k_{44} \text{ and } k_{ij} =0 \text{ for } 1\leq i\leq 4,3\leq j\leq 4, i<j.$$
Hence  $K=k_{33}x_3^2+k_{44}x_4^2.$ Comparing the degree 3 terms in \eqref{kolmogorov3-s12}, we get $\sum\limits_{i=3}^4 \psi_i^{(1)}x_i^2=0$ which follows that ${\psi_3^{(1)}=\psi_4^{(1)}=0}$, since these are linear polynomials. Now comparing the degree 4 terms in \eqref{kolmogorov3-s12}, we get
\begin{dmath*}
    -4a^2(\sum_{i=1}^2 \psi_i^{(2)}x_i^2)+2(\sum_{i=3}^4 \psi_i^{(2)}x_i^2)=(k_{33}x_3^2+k_{44}x_4^2)(-2a^2(x_1^2+x_2^2)+x_3^2+x_4^2).
\end{dmath*}
Then using \eqref{homogeneous-2-kolmogorov3-s12} in the above equation, we have
\begin{dmath*}
    \sum_{i=3}^4 \psi_i^{(2)}x_i^2=\frac{1}{2}(k_{33}x_3^2+k_{44}x_4^2)(-a^2(x_1^2+x_2^2)+x_3^2+x_4^2).
\end{dmath*}
This implies $$(\psi_3^{(2)}-\frac{1}{2}k_{33}(-a^2(x_1^2+x_2^2)+x_3^2+x_4^2))x_3^2+(\psi_4^{(2)}-\frac{1}{2}k_{44}(-a^2(x_1^2+x_2^2)+x_3^2+x_4^2))x_4^2=0.$$
Therefore, by Lemma \ref{polynomial-ch-s12},
$$\psi_3^{(2)}=\frac{1}{2}k_{33}(-a^2(x_1^2+x_2^2)+x_3^2+x_4^2)+\beta x_4^2\quad \text{and} \quad \psi_4^{(2)}=\frac{1}{2}k_{44}(-a^2(x_1^2+x_2^2)+x_3^2+x_4^2)-\beta x_3^2$$
for some $\beta \in \RR$. Hence, $\psi_3,\psi_4$ are completely determined. 

 If $P_1, \ldots, P_4$ are given by \eqref{kolmogorov-form-s12}, then they satisfy \eqref{vectorfield-s12}. Thus, the converse part is true.
\end{proof}
\begin{theorem}\label{thm:cubic-ham-s12}
    There is no Hamiltonian of a cubic Kolmogorov vector field on $S^1\times S^2$. However, every cubic Kolmogorov vector field on $S^1\times S^2$ has a rational first integral.
\end{theorem}
\begin{proof}
Any cubic Kolmogorov vector field on $S^1\times S^2$ can be described by \eqref{kolmogorov-form-s12}. Suppose $H$ is a Hamiltonian of the vector field \eqref{kolmogorov-form-s12}. Then $$H=x_1x_2(\sum\limits_{i+j=0}^2 a_{ij}x_1^i x_2^j)+x_3x_4(\sum\limits_{i+j=0}^2 b_{ij}x_3^i x_4^j)+c_0x_1x_2x_3x_4$$ by \cite[Theorem 1]{LlXi17}. Since $-\frac{\partial H}{\partial x_2}=P_1$, we get 
$$-x_1(\sum\limits_{i+j=0}^2 a_{ij}x_1^i x_2^j)-x_1x_2(a_{01}+a_{11}x_1+2a_{02}x_2)-c_0x_1x_3x_4=x_1(\frac{1}{4}K+\alpha x_2^2).$$
Thus, $K=a_{00}=a_{10}=a_{01}=a_{11}=a_{20}=c_0=0$. Hence, $H=a_{02}x_1x_2^3+x_3x_4(\sum\limits_{i+j=0}^2 b_{ij}x_3^i x_4^j)$. Again, $-\frac{\partial H}{\partial x_4}=P_3$ implies 
$$-x_3(\sum\limits_{i+j=0}^2 b_{ij}x_3^i x_4^j)-x_3x_4(b_{01}+b_{11}x_3+2b_{02}x_4)=\beta x_3x_4^2.$$ Then $b_{00}=b_{10}=b_{01}=b_{11}=b_{20}=0$. Hence, $H=a_{02}x_1x_2^3+b_{02}x_3x_4^3$. Now from $\frac{\partial H}{\partial x_1}=P_2=-\alpha x_1^2x_2$ and $\frac{\partial H}{\partial x_3}=P_4=-\beta x_3^2x_4$, we obtain that $H=0$, which is absurd. Hence, a cubic Kolmogorov vector field on $S^1\times S^2$ cannot have any Hamiltonian.

 Let $\chi=(P_1, ..., P_4)$ be a vector field of the form \eqref{kolmogorov-form-s12}. Then the hypesurfaces given by $(x_1^2+x_2^2-a^2)^2+x_3^2+x_4^2-1=0$ and  $x_1^2+x_2^2=0$ are invariant algebraic hypersurfaces of $\chi$ with cofactors $K$ and $\frac{1}{2}K$ respectively.
Hence by Darboux Integrability Theory (\cite[Theorem 5]{LlZh02}), the function
$((x_1^2+x_2^2-a^2)^2+x_3^2+x_4^2-1)(x_1^2+x_2^2)^{-2}$ is a rational first integral of $\chi$.
\end{proof}

\begin{lemma}\label{deg-m-n-s12}
Suppose $\chi=(P_1,...,P_4)$ is a polynomial vector field on $S^1\times S^2$ where $P_i$ are homogeneous polynomials with $\deg(P_1)=\deg(P_2)=m$ and $\deg(P_3)=\deg(P_4)=n$. If $m-1\leq n\leq m+3$ then $G=(x_1^2+x_2^2-a^2)^2 + x_3^2 + x_4^2 -1$ is a first integral of $\chi$.
\end{lemma}

\begin{proof} Each monomial in the expression of $\chi G$ has degree either $m+3$, $m+1$ or $n+1$. If $\chi$ is a vector field on $S^1\times S^2$ then $\chi G=KG=K((x_1^2 +x_2^2)^2 -2a^2(x_1^2 +x_2^2)+x_3^2 +x_4^2) +K(a^4-1)$ for some $K\in \RR[x_1,...,x_4]$.

If $m-1\leq n\leq m+2$ then $\chi G$ has degree at most $m+3$. So $K$ has the degree at most $m-1$. Hence $K(a^4 -1)$ has degree at most $m-1$. However, the degree of each monomial in the expression of $\chi G$ is at least $m$. Hence $K$ must be zero.

If $n=m+3$ then $\chi G$ has degree $m+4$. So $K$ has degree $m$. Hence $K(a^4 -1)$ has degree $m$. However, the degree of each monomial in the expression of $\chi G$ is at least $m+1$. Hence $K$ must be zero. 
Therefore, $G$ must be a first integral of $\chi$, since in each case above $K=0$.
\end{proof}
\begin{theorem}
\label{type-n-s12}
Let $\chi=(P_1,...,P_4)$ be a Type-$n$ vector field in $\RR^4$. Then $\chi$ is a Type-$n$ vector field on $S^1\times S^2$ if and only if  there exists $A,B\in \RR[x_1,...,x_4]$ such that
\begin{equation}\label{type-n-form-s12}
P_1=Ax_2,P_2=-Ax_1, P_3=Bx_4,P_4=-Bx_3.
\end{equation}
\end{theorem}
\begin{proof}
By Lemma \ref{deg-m-n-s12}, any Type-$n$ vector field satisfies the following.
\begin{equation*}\label{on-s12}
\begin{split}
     & 4(x_1^2+x_2^2-a^2)(P_1x_1 + P_2x_2)+2(P_3x_3 + P_4x_4)=0\\
     \implies& 4(x_1^2+x_2^2)(P_1x_1 + P_2x_2)-4a^2(P_1x_1 + P_2x_2)+2(P_3x_3 + P_4x_4)=0.
\end{split}
     \end{equation*}
     Comparing the same degree parts in the above equation, we get $4(x_1^2+x_2^2)(P_1x_1 + P_2x_2)=0$ and $-4a^2(P_1x_1 + P_2x_2)+2(P_3x_3 + P_4x_4)=0$. Then, $P_1x_1 + P_2x_2=0$ since $x_1^2+x_2^2\neq 0$. Hence $P_3x_3+P_4x_4=0$.
     Now by Lemma \ref{polynomial-ch-s12}, $P_1=Ax_2,P_2=-Ax_1,P_3=Bx_4$ and $P_4=-Bx_3$ for some $A,B\in \RR[x_1,...,x_4]$.

If $P_1,...,P_4$ are given by \eqref{type-n-form-s12}, then they satisfy \eqref{vectorfield-s12}. So, the converse part follows.
\end{proof}
\begin{remark}
    Similar to the above proof, using the argument of separating the same degree terms of $\chi G$, it can be proved that if $\chi=(P_1,...,P_4)$ is a polynomial vector field in $\RR^4$ satisfying the hypothesis of Lemma \ref{deg-m-n-s12}, it will be a vector field on $S^1\times S^2$ if and only if there exists $A,B\in \RR[x_1,...,x_4]$ such that 
    $$P_1=Ax_2,P_2=-Ax_1, P_3=Bx_4,P_4=-Bx_3.$$
\end{remark}
\begin{proposition}
There exists no Pseudo Type-$n$ vector field on $S^1\times S^2$.
\end{proposition}
\begin{proof}
Let $\chi=(P_1, P_2, P_3, 0)$ be a Pseudo Type-$n$ vector field on $S^1\times S^2$. Thus,
$$4(x_1^2+x_2^2-a^2)(P_1x_1 + P_2x_2)+2P_3x_3=K((x_1^2+x_2^2-a^2)^2+x_3^2+x_4^2-1)$$
for some $K \in \RR[x_1, ..., x_4]$ with $\deg K \leq n-1$.
So, $K=0$, since the degree of each term on the left side is greater than $n$. Then, $4(x_1^2+x_2^2-a^2)(P_1x_1 + P_2x_2)+2P_3x_3=0$. Thus, $x_1^2+x_2^2-a^2$ is a factor of $P_3$. This is not possible, as $P_3$ is homogeneous and $a \neq 0$. 

\end{proof}
\begin{proposition}
    Any Type-$1$ vector field on $S^1\times S^2$ is Hamiltonian.
\end{proposition}
\begin{proof}
    Suppose $\chi$ is a Type-1 vector field on $S^1\times S^2$. Then, by Theorem \eqref{type-n-s12}, $$P_1=Ax_2,P_2=-Ax_1,P_3=Bx_4,P_4=-Bx_3\quad \text{for some } A,B\in \RR.$$ One can check that $H=-\frac{A}{2}(x_1^2+x_2^2)-\frac{B}{2}(x_3^2+x_4^2)$ is a Hamiltonian for $\chi$.
\end{proof}
\begin{theorem}\label{thm:degree-relation}
    Let $P_i\in \mathbb{R}[x_1,...,x_4]$ be non-zero homogeneous polynomials such that ${\deg(P_i)=m_i}$ for $1\leq i \leq 4$. Assume that $G=(x_1^2+x_2^2-a^2)^2+x_3^2+x_4^2-1$ is a first integral of $\chi=(P_1,...,P_4)$. Then the following holds.
    \begin{enumerate}[(a)]
        \item If $m_1\neq m_2$ then $\max(m_1,m_2)=\min(m_1,m_2)+2$ and 
        
        $\{m_3,m_4\}=\{\max(m_1,m_2)+2,\min(m_1,m_2)\}$.
        \item If $m_1=m_2$ then either $m_3=m_4$ or $\{m_3,m_4\}=\{m_1,m_1+2\}$
    \end{enumerate}
\end{theorem}
\begin{proof}
Since $G$ is a first integral of $\chi$, $\chi G=0$, i.e., 
\begin{equation}\label{first-integral-s12}
4(x_1^2+x_2^2-a^2)(P_1x_1+P_2x_2)+2(P_3x_3+P_4x_4)=0.
\end{equation}
\begin{enumerate}[(a)]
    \item Since $m_1\neq m_2$, $P_1x_1+P_2x_2$ is a non-zero polynomial. Assume that $m_1<m_2$. So either of $m_3,m_4$ should be $m_2+2$. Without loss of generality, suppose $m_3=m_2+2$. We rewrite $\chi G$ by separating the same degree homogeneous parts as follows. $$\chi G=(4(x_1^2+x_2^2)P_2x_2+2P_3x_3)+4(x_1^2+x_2^2)P_1x_1-4a^2P_1x_1-4a^2P_2x_2+2P_4x_4.$$ Notice that
    \begin{equation*}
    \begin{split}
        \deg(4(x_1^2+x_2^2)P_2x_2+2P_3x_3)=m_2+3, \deg(((x_1^2+x_2^2)P_1x_1)=m_1+3 ,\deg(P_1x_1)=m_1+1,\\
        \deg(P_2x_2)=m_2+1 \text{ and }\deg(P_4x_4)=m_4+1.
            \end{split}
    \end{equation*}
    In the next step, we will match the degrees of the terms in $\chi G=0$ under certain conditions.
\begin{itemize}
    \item \underline{$m_4=m_2+2$:} Then there are no term of degree $m_1+1$ in $\chi G$ except $P_1x_1$. Hence $P_1=0$, which is not possible.
    \item \underline{$m_4>m_2+2$:} Then every term of $\chi G$ has strictly lower degree than $P_4x_4$. Hence $P_4=0$, which is not possible.
    \item \underline{$m_4<m_2+2:$} Then $m_4\in \{m_1,m_2,m_1+2\}$. If $m_4=m_1$ then $m_1+3$ must be $m_2+1$, i.e., $m_2=m_1+2$. If $m_4=m_2$ or $m_1+2$, then $P_1x_1=0$ since $P_1x_1$ will be the only lowest degree term in $\chi G$. Hence $P_1=0$, which is not possible.
\end{itemize}
So, in this case, $m_2=m_1+2, m_3=m_2+2, m_4=m_1$.
\item Now consider $m_1=m_2$. If $P_1x_1+P_2x_2=0$ then $P_3y+P_4z=0$ from \eqref{first-integral-s12}, which implies $m_3=m_4$.
\par Suppose $P_1x_1+P_2x_2\neq 0$. We rewrite $\chi G$ by separating same degree homogeneous parts as follows. $$\chi G=4(x_1^2+x_2^2)(P_1x_1+P_2x_2)-4a^2(P_1x_1+P_2x_2)+2P_3x_3+2P_4x_4.$$
Notice that
\begin{equation*}
\begin{split}
\deg (4(x_1^2+x_2^2)(P_1x_1+P_2x_2)=m_1+3, \deg(-4a^2(P_1x_1+P_2x_2))=m_1+1,\\ \deg(2P_3x_3)=m_3+1, \deg(2P_4x_4)=m_4+1.
\end{split}
\end{equation*}
Observe that $m_1+3$ must be either $m_3+1$ or $m_4+1$. Otherwise, from $\chi G=0$, we obtain $P_1x_1+P_2x_2=0$, which contradicts our assumption. Without loss of generality, suppose $m_1+3=m_3+1$. Then $m_1+1=m_4+1$, since $P_1x_1+P_2x_2$ and $P_4$ are nonzero.
\end{enumerate}
\end{proof}

\begin{theorem}\label{thm:2-first-integral-s12}
    A Type-$n$ vector field on $S^1\times S^2$ has at least two independent first integrals.
\end{theorem}
\begin{proof}
    Suppose $f$ is a first integral of the Type-$n$ vector field \eqref{type-n-form-s12} on $S^1\times S^2$. Then $\sum\limits_{i=1}^4 P_i \frac{\partial f}{\partial x_i}=0$. This implies
    \begin{equation*}
        Ax_2\frac{\partial f}{\partial x_1}-Ax_1\frac{\partial f}{\partial x_2}+Bx_4\frac{\partial f}{\partial x_3}-Bx_3\frac{\partial f}{\partial x_4}=0.
\end{equation*}
By \cite[Theorem 3(Page 53)]{Sn13}, solutions of the above can be given in terms of the solutions of the equations
$$\frac{dx_1}{Ax_2}=\frac{dx_2}{-Ax_1}=\frac{dx_3}{Bx_4}=\frac{dx_4}{-Bx_3}=\frac{df}{0}=\frac{x_1 dx_1+x_2 dx_2}{0}=\frac{x_3dx_3+x_4dx_4}{0}.$$
Hence, we obtain two independent first integrals $x_1^2+x_2^2$ and $x_3^2+x_4^2$.
\end{proof}
\begin{corollary}
    Any Type-$n$ Hamiltonian vector field on $S^1\times S^2$ is integrable in the Liouville sense.
\end{corollary}
\section{Polynomial vector fields on $S^2\times S^1$}\label{sec:pvf_s2s1}
In this section, we study and characterize polynomial vector fields on $S^2\times S^1$. Recall the presentation of $S^2\times S^1\subset \RR^4$ from \eqref{eq:s21}. A vector field $\chi=(P_1,...,P_4)$ in $\RR^4$ is a vector field on $S^2\times S^1$ if
\begin{equation}\label{vectorfield-s21}
    4(x_1^2+x_2^2+x_3^2-b^2)(P_1x_1 + P_2x_2 +P_3x_3)+2P_4x_4=K((x_1^2+x_2^2+x_3^2-b^2)^2 + x_4^2 -1)
\end{equation}
for some $K\in \RR[x_1,...,x_4]$.
\begin{lemma}\label{polynomial-ch-s21}
    Let $n\in \NN$ and $Q_1,Q_2,Q_3\in \RR[x_1,...,x_4]$ be non-zero polynomials such that $Q_1x_i^n+Q_2x_j^n+Q_3x_k^n$ is the zero polynomial for distinct $i,j,k\in \{1,...,4\}$. Then $Q_1=Ax_j^n+Bx_k^n,Q_2=-Ax_i^n + Cx_k^n$ and $Q_3=-Bx_i^n-Cx_j^n$ for some polynomials $A, B, C.$
\end{lemma}
\begin{proof}
The equation     
$Q_1x_i^n+Q_2x_j^n+Q_3x_k^n=0$ implies that $Q_1=\frac{Q_2x_j^n+Q_3x_k^n}{x_i^n}$. Hence we get $Q_1=Ax_j^n+Bx_k^n$ for some polynomials $A$ and $B$ since $x_i,x_j$ and $x_k$ are mutually prime and any monomial in $Q_2x_j^n+Q_3x_k^n$ has power of both $x_j$ and $x_k$ not less than $n$. Similarly,  $Q_2=A^{'}x_i^n+B^{'}x_k^n$ and $Q_3=A^{''}x_i^n+B^{''}x_j^n$.

    Now $Q_3x_k^n=-Q_1x_i^n-Q_2x_j^n=-Ax_i^n x_j^n-Bx_i^n x_k^n -A^{'}x_i^nx_j^n-B^{'}x_j^nx_k^n$. This implies $x_k^n$ divides $A+A^{'}$ and hence $A+A^{'}=Lx_k^n$ for some polynomial $L \in \RR[x_1, ..., x_4]$. So, $Q_2 =-Ax_i^n+(Lx_i^n+B^{'})x_k^n$.
    
    Again $Q_2x_j^n=-Q_1x_i^n-Q_3x_k^n=-Ax_i^nx_j^n-Bx_i^nx_k^n-A^{''}x_i^nx_k^n-B^{''}x_j^nx_k^n$. This implies that $x_j^n$ divides $A^{''}+B$ and hence $A^{''}+B=L^{'}x_j^n$ for some polynomial $L^{'}$. Hence, we have ${Q_3=-Bx_i^n+(L^{'}x_i^n+B^{''})x_j^n}$. Substituting these values of $Q_1,Q_2$ and $Q_3$ in the equation ${Q_1x_i^n+Q_2x_j^n+Q_3x_k^n=0}$, we get
\begin{equation*}\label{on-s12}
\begin{split}
     & (Ax_j^n + Bx_k^n)x_i^n +(-Ax_i^n+ (x_i^nL+B')x_k^n)x_j^n +(-Bx_i^n +(x_i^n L' +B'')x_j^n)x_k^n=0\\
     \implies& L^{'}x_i^n+B^{''}=-(Lx_i^n+B^{'})=0.
\end{split}
     \end{equation*}
 Hence, assuming $C := Lx_i^n+B^{'}$, we get our result.
\end{proof}
\begin{theorem}
\label{quadratic-thm-s21}
    Let $\chi=(P_1,...,P_4)$ be a quadratic vector field in $\RR^4$. Then $\chi$ is a quadratic vector field on $S^2\times S^1$ if and only if
    \begin{equation}\label{quadratic-ch-s21}
        \begin{split}
            P_1&=\frac{c}{4}x_1x_4+f x_2+g x_3,\\
            P_2&=\frac{c}{4}x_2x_4-f x_1+h x_3,\\
            P_3&=\frac{c}{4}x_3x_4-g x_1-h x_2, \text{ and}\\
            P_4&=\frac{c}{2}(-b^2(x_1^2+x_2^2+x_3^2)+x_4^2+b^4-1)
        \end{split}
    \end{equation}
    where $f,g,h$ are linear polynomials and $c\in \RR$. Moreover, $K=c x_4$ is the cofactor of $\chi$ for $S^2\times S^1$.
\end{theorem}
\begin{proof}
Suppose $\chi$ is a quadratic vector field on $S^2\times S^1$. So it must satisfy
    \begin{dmath}\label{quadratic-s21}
        4(x_1^2+x_2^2+x_3^2-b^2)(\sum_{i=1}^3 P_ix_i)+2P_4x_4=(\sum\limits_{j=1}^4 k_jx_j)((x_1^2+x_2^2+x_3^2-b^2)^2 + x_4^2 -1)
    \end{dmath} since there is no constant term on the left side. We write 
    \begin{equation*}
    P_i=P_i^{(2)}+P_i^{(1)}+P_i^{(0)}
    \end{equation*}
    as in \eqref{homogeneous-division}. Comparing degree 5 terms in \eqref{quadratic-s21}, we get
    \begin{dmath*}
        4(x_1^2+x_2^2+x_3^2)(\sum_{i=1}^3 P_i^{(2)}x_i)=(\sum\limits_{j=1}^4 k_jx_j)(x_1^2+x_2^2+x_3)^2.
    \end{dmath*}
    Hence
    \begin{dmath}\label{deg-5-s21}
        \sum_{i=1}^3 P_i^{(2)}x_i=\frac{1}{4}(\sum\limits_{j=1}^4 k_jx_j)(x_1^2+x_2^2+x_3^2).
    \end{dmath}
    Also comparing degree $4$ terms in \eqref{quadratic-s21}, we get
    \begin{dmath*}
        4(x_1^2+x_2^2+x_3^2)(\sum_{i=1}^3 P_i^{(1)}x_i)=0.
    \end{dmath*}
    Hence $\sum_{i=1}^3 P_i^{(1)}x_i=0$. Then, by Lemma \ref{polynomial-ch-s21}, $P_1^{(1)}=\alpha x_2+\beta x_3, P_2^{(1)}=-\alpha x_1+\gamma x_3$ and $P_3^{(1)}=-\beta x_1-\gamma x_2$ for some $\alpha,\beta,\gamma \in \RR$.
    
    In particular, if $x_4=0$, then \eqref{quadratic-s21} gives the following.
    \begin{dmath}\label{x_4-0-s21}
         4(x_1^2+x_2^2+x_3^2-b^2)(\sum_{i=1}^3 P_i(x_1,x_2,x_3,0)x_i)=(\sum_{j=1}^3 k_jx_j)((x_1^2+x_2^2+x_3^2-b^2)^2 -1).
    \end{dmath}
    Comparing degree 3 terms in \eqref{x_4-0-s21},
\begin{dmath*}
    4(\sum_{i=1}^3 x_i^2)(\sum_{i=1}^3 P_i^{(0)}(x_1,x_2,x_3,0)x_i)-4b^2(\sum_{i=1}^3 P_i^{(2)}(x_1,x_2,x_3,0)x_i)=-2b^2(\sum_{j=1}^3 k_jx_j)(\sum_{i=1}^3 x_i^2).
\end{dmath*}
Then, for $x_4=0$, the above equation together with \eqref{deg-5-s21} gives the following.
\begin{dmath}
\label{x_4-0-s21-quad}
    4(\sum_{i=1}^3 P_i^{(0)}x_i)=-b^2(\sum_{j=1}^3 k_jx_j).
\end{dmath}
Also comparing degree 1 terms in \eqref{x_4-0-s21}, we get
\begin{dmath}\label{x_4-0-s21-linear}
    -4b^2(\sum_{i=1}^3 P_i^{(0)}x_i)=(b^4-1)(\sum_{j=1}^3 k_jx_j).
\end{dmath}
Now, from \eqref{x_4-0-s21-quad} and \eqref{x_4-0-s21-linear}, we obtain $k_1=k_2=k_3=0$. Hence $P_1^{(0)}=P_2^{(0)}=P_3^{(0)}=0$. Therefore, from \eqref{deg-5-s21} we get the following.
\begin{dmath}\label{deg2-homogeneous-s21}
\sum_{i=1}^3 P_i^{(2)}x_i=\frac{1}{4}k_4x_4(x_1^2+x_2^2+x_3^2).
\end{dmath}
Then, by Lemma \ref{polynomial-ch-s21}, $P_1^{(2)}=\frac{1}{4}k_4x_1x_4+Ax_2+Bx_3, P_2^{(2)}=\frac{1}{4}k_4x_2x_4-Ax_1+Cx_3$ and $ {P_3^{(2)}=\frac{1}{4}k_4x_3x_4-Bx_1-Cx_2}$ where $A,B,C$ are linear homogeneous polynomials.
So we have found all the homogeneous parts of $P_1,P_2$ and $P_3$. 

Let us now find $P_4$.
Comparing degree 1 terms in \eqref{quadratic-s21}, we get
$$P_4^{(0)}=\frac{b^4 -1}{2}k_4.$$
Comparing degree 2 terms in \eqref{quadratic-s21}, we get $P_4^{(1)}x_4=0$ which follows that $P_4^{(1)}=0$. Now, comparing degree 3 terms in \eqref{quadratic-s21}, we get
\begin{dmath*}
-4b^2(\sum_{i=1}^3 P_i^{(2)}x_i)+2P_4^{(2)}x_4=k_4x_4(-2b^2(x_1^2+x_2^2+x_3^2)+x_4^2).
\end{dmath*}
Then using \eqref{deg2-homogeneous-s21} in the above equation, we obtain
\begin{dmath*}
    P_4^{(2)}x_4=\frac{1}{2}k_4x_4(-b^2(x_1^2+x_2^2+x_3^2)+x_4^2).
\end{dmath*}
Hence $P_4^{(2)}=\frac{k_4}{2}(-b^2(x_1^2+x_2^2+x_3^2)+x_4^2).$
So, we have found all the homogeneous components of $P_4$. Renaming $k_4$ as $c$, we get our result.

If $P_1,...,P_4$ are given by \eqref{quadratic-ch-s21}, then they satisfy \eqref{vectorfield-s21}. Thus, the converse part is true.
\end{proof}
\begin{corollary}\label{linear-s21}
    $\chi=(P_1,...,P_4)$ is a linear vector field on $S^2\times S^1$ if and only if there exists $\alpha,\beta,\gamma\in \RR$ such that
    $$P_1=\alpha x_2+\beta x_3, P_2=-\alpha x_1+\gamma x_3, P_3=-\beta x_1 -\gamma x_2, \text{ and } P_4=0.$$
\end{corollary}
Observe that if $P_1,P_2,P_3\neq 0$ then a linear vector field $\chi=(P_1,...,P_4)$ on $S^2\times S^1$ is a Pseudo Type-1 vector field.

\begin{proof}
    The vector field \eqref{quadratic-ch-s21} on $S^2\times S^1$ is a linear if and only if $c=0$ and $f,g,h$ are constants. Hence the result follows.
\end{proof}
\begin{corollary}\label{cor_lvvf_s2s1}
There is no Lotka-Volterra vector field on $S^2\times S^1$.
\end{corollary}
\begin{proof}
    Suppose $\chi=(P_1,...,P_4)$ defines a Lotka-Volterra vector field on $S^2\times S^1$. Then $x_i$ divides $P_i$ for $i\in \{1,...,4\}$. Since $\chi$ is a quadratic vector field, it will be of the form \eqref{quadratic-ch-s21}. $x_4$ divides $P_4$ only when $c=0$. In that case, $P_4=0, P_1=fx_2+gx_3, P_2=-fx_1+hx_3$ and $P_3=-gx_1-hx_2$.
    
    Suppose $P_i=\ell_ix_i$ for $i\in \{1,2,3\}$. Then $\sum\limits_{i=1}^3 P_ix_i=0=\sum\limits_{i=1}^3 \ell_ix_i^2.$ Since $\ell_i$ are linear polynomials, $\ell_i=0$ for each $i\in \{1,2,3\}$. Hence, $P_1=P_2=P_3=0$.
\end{proof}
\begin{theorem}\label{thm:noHam_quad}
There is no Hamiltonian of a quadratic vector field on $S^2\times S^1$. However, every quadratic vector field on $S^2\times S^1$ has a rational first integral.
\end{theorem}
\begin{proof}
Any quadratic vector field on $S^2\times S^1$ can be described by \eqref{quadratic-ch-s21}. Suppose $H$ is a Hamiltonian of the quadratic vector field \eqref{quadratic-ch-s21} on $S^2\times S^1$, where $f=f_0+\sum\limits_{i=1}^4 f_ix_i,\break g=g_0+\sum\limits_{i=1}^4 g_ix_i$ and $h=h_0+\sum\limits_{i=1}^4 h_ix_i; f_i,g_i,h_i\in \RR$. 

Since $\frac{\partial H}{\partial x_3}=P_4=\frac{c}{2}(-b^2(x_1^2+x_2^2+x_3^2)+x_4^2+b^4-1)$, we get $$H=\frac{c}{2}(-b^2(x_1^2+x_2^2)+x_4^2+b^4-1)x_3-\frac{b^2c}{6}x_3^3+H_1(x_1,x_2,x_4)$$
where $H_1$ is a function independent of $x_3$. Now, from $-\frac{\partial H}{\partial x_4}=P_3$, we get
$$-cx_3x_4-\frac{\partial H_1}{\partial x_4}=\frac{c}{4}x_3x_4-g x_1-h x_2.$$
Thus,
\begin{equation}\label{eq:hamiltonian-1}
    c=g_3=h_3=0 \quad \text{and}
\end{equation}
since $H_1$ is a function independent of $x_3$. So, $\frac{\partial H_1}{\partial x_4}=(\sum\limits_{\substack{i=0, i\neq 3}}^4 g_ix_i)x_1+(\sum\limits_{\substack{i=0, i\neq 3}}^4 h_ix_i)x_2$ and hence
$$H_1=(g_0+g_1x_1+g_2x_2+\frac{1}{2}g_4x_4)x_1x_4+(h_0+h_1x_1+h_2x_2+\frac{1}{2}h_4x_4)x_2x_4+H_2(x_1,x_2)$$
where $H_2$ is a function of $x_1$ and $x_2$ only. Notice that $H=H_1$ since $c=0$. As $\frac{\partial H}{\partial x_1}=P_2$, we get
$$(g_0+2g_1x_1+g_2x_2+\frac{1}{2}g_4x_4)x_4+h_1x_2x_4+\frac{\partial H_2}{\partial x_1}=-fx_1+hx_3.$$
Thus, 
\begin{dmath*}
    (g_0+\frac{1}{2}g_4x_4)x_4+(2g_1+f_4)x_1x_4+(g_2+h_1)x_2x_4+(f_3-h_1)x_1x_3-(\sum\limits_{\substack{i=0, i\neq 1,3}}^4 h_ix_i)x_3=-\frac{\partial H_2}{\partial x_1}-(f_0+f_1x_1+f_2x_2)x_1
\end{dmath*} Then
\begin{equation}\label{eq:hamiltonian-2}
    g_0=g_4=h_0=h_2=h_4=0, f_4=-2g_1, f_3=-g_2=h_1,
\end{equation}
since $H_2$ is a function of $x_1,x_2$ only. Therefore,
$$H_2=-\frac{1}{2}(f_0+f_2x_2)x_1^2-\frac{1}{3}f_1x_1^3+H_3(x_2)$$
where $H_3$ is a function of $x_2$ only. Hence,
$$H=H_1=-\frac{1}{2}(f_0+f_2x_2)x_1^2-\frac{1}{3}f_1x_1^3+g_1x_1^2x_4+H_3(x_2).$$
Again, $-\frac{\partial H}{\partial x_2}=P_1$ gives 
$\frac{1}{2}f_2x_1^2-H_3^{'} =fx_2+gx_3$. Thus
\begin{equation}\label{eq:hamiltonian-3}
    f_1=f_2=f_4=g_1=0,
\end{equation}
since $H_3$ is a function of $x_2$ only.
Combining \eqref{eq:hamiltonian-1}, \eqref{eq:hamiltonian-2} and \eqref{eq:hamiltonian-3}, we obtain ${f_3=-g_2=h_1}$ and other coefficients of $f,g$ and $h$ are zero except $f_0$. Hence, $f=f_0+h_1x_3,$ $g=-h_1x_2, $ $h=h_1x_1$. For these $f,g$ and $h$ the vector field \eqref{quadratic-ch-s21} becomes a linear, which contradicts the assumption. So, there is no Hamiltonian of the quadratic vector field \eqref{quadratic-ch-s21}.
 
The hypersurfaces given by $x_1^2+x_2^2+x_3^2=0$ and ${(x_1^2+x_2^2+x_3^2-b^2)^2+x_4^2-1=0}$ are two invariant algebraic hypersurfaces of the vector field \eqref{quadratic-ch-s21} with cofactors $\frac{c}{2}x_4$ and $cx_4$ respectively. By Darboux Integrability Theory (Theorem 5, \cite{LlZh02}), the rational function $(x_1^2+x_2^2+x_3^2)^{-2} ((x_1^2+x_2^2+x_3^2-b^2)^2+x_4^2-1)$ is a first integral of the vector field.
\end{proof}

\begin{theorem}\label{thm:kolmo_vf_s2s1}
    Let $\chi=(P_1,...,P_4)$ be a cubic Kolmogorov vector field in $\RR^4$. Then $\chi$ is a vector field on $S^2\times S^1$ if and only if
    \begin{equation}\label{kolmogorov-form-s21}
        \begin{split}
     P_1&=x_1(\frac{c}{4}x_4^2 +\alpha x_2^2 +\beta x_3^2),\\
     P_2&=x_2(\frac{c}{4}x_4^2 -\alpha x_1^2+ \gamma x_3^2),\\
     P_3&=x_3(\frac{c}{4}x_4^2 -\beta x_1^2- \gamma x_2^2), \text{ and}\\
     P_4&=\frac{1}{2}cx_4(-b^2(x_1^2+x_2^2+x_3^2)+x_4^2+b^4-1)
        \end{split}
    \end{equation}
    where $c,\alpha,\beta,\gamma\in \RR$. Moreover, $K=cx_4^2$ is the cofactor of $\chi$ for $S^2\times S^1$.
\end{theorem}
\begin{proof}
Consider a cubic Kolmogorov vector field $\chi=(P_1,...,P_4)$ in $\RR^4$ with $$P_i=x_i \psi_i$$ for $i=1,...,4$. 
 We write
    $$\psi_i=\psi_i^{(2)}+\psi_i^{(1)}+\psi_i^{(0)}$$ as in \eqref{homogeneous-division}.
    \par Suppose $\chi$ is a vector field on $S^2\times S^1$. So it must satisfy
    \begin{dmath}\label{kolmogorov3-s21}
        4(\sum_{i=1}^3 x_i^2-b^2)(\sum_{i=1}^3 \psi_ix_i^2)+2 \psi_4x_4^2=(\sum_{1\leq i\leq j\leq 4} k_{ij}x_ix_j)((\sum_{i=1}^3 x_i^2-b^2)^2 + x_4^2 -1)
    \end{dmath} since there is no constant term on the left side. Comparing degree 6 terms in \eqref{kolmogorov3-s21}, we get
    \begin{dmath*}
        4(\sum_{i=1}^3 x_i^2)(\sum_{i=1}^3 \psi_i^{(2)}x_i^2)=(\sum_{1\leq i\leq j\leq 4} k_{ij}x_ix_j)(\sum_{i=1}^3 x_i^2)^2.
    \end{dmath*}
    Hence
    \begin{dmath}\label{kolmogorov3-deg-6-s21}
        \sum_{i=1}^3 \psi_1^{(2)}x_i^2=\frac{1}{4}(\sum_{1\leq i\leq j\leq 4} k_{ij}x_ix_j)(\sum_{i=1}^3 x_i^2).
    \end{dmath}
    Also comparing degree $5$ terms in \eqref{kolmogorov3-s21}, we get
    \begin{dmath*}
        4(\sum_{i=1}^3 x_i^2)(\sum_{i=1}^3 \psi_i^{(1)}x_i^2)=0.
    \end{dmath*}
    Hence $\sum_{i=1}^3 \psi_i^{(1)}x_i^2=0$. Then $\psi_1^{(1)}, \psi_2^{(1)}, \psi_3^{(1)}$ are zero since they are degree 1 polynomials.
    
    In particular, if $x_4=0$ then \eqref{kolmogorov3-s21} gives the following.
    \begin{dmath}\label{kolmogorov3-x_4-0-s21}
         4(\sum_{i=1}^3 x_i^2-b^2)(\sum_{i=1}^3 \psi_i(x_1,x_2,x_3,0)x_i^2)=(\sum_{1\leq i\leq j\leq 3} k_{ij}x_ix_j)((\sum_{i=1}^3 x_i^2-b^2)^2 -1).
    \end{dmath}
    Comparing degree 4 terms in \eqref{kolmogorov3-x_4-0-s21}
    \begin{dmath*}
4(\sum_{i=1}^3 x_i^2)(\sum_{i=1}^3 \psi_i^{(0)}(x_1,x_2,x_3,0)x_i^2)-4b^2(\sum_{i=1}^3 \psi_i^{(2)}(x_1,x_2,x_3,0)x_i^2)=-2b^2(\sum_{1\leq i\leq j\leq 3} k_{ij}x_ix_j)(\sum_{i=1}^3 x_i^2).
\end{dmath*}
Then, for $x_4=0$, the above equation together with \eqref{kolmogorov3-deg-6-s21} gives the following.
\begin{dmath}
\label{x_4-0-s21-kolmogorov-quad}
    4(\sum_{i=1}^3 \psi_i^{(0)}x_i^2)=-b^2(\sum_{1\leq i\leq j\leq 3} k_{ij}x_ix_j).
\end{dmath}
Also comparing degree 2 terms in \eqref{kolmogorov3-x_4-0-s21}, we get
\begin{dmath}
\label{x_4-0-s21-kolmogorov-linear}
    -4b^2(\sum_{i=1}^3 \psi_i^{(0)}x_i^2)=(b^4-1)(\sum_{1\leq i\leq j\leq 3} k_{ij}x_ix_j).
\end{dmath}
Now, from \eqref{x_4-0-s21-kolmogorov-quad} and \eqref{x_4-0-s21-kolmogorov-linear}, we obtain $k_{ij}=0$ for $1\leq i\leq j\leq 3$. Hence we get, $\psi_1^{(0)}=\psi_2^{(0)}=\psi_3^{(0)}=0$. Therefore, from \eqref{kolmogorov3-deg-6-s21} we get the following.
\begin{dmath}\label{homogeneous-2-kolmogorov3-s21}
\sum_{i=1}^3 \psi_i^{(2)}x_i^2=\frac{1}{4}(\sum_{i=1}^4 k_{i4}x_ix_4)(\sum_{i=1}^3 x_i^2).
\end{dmath}
Taking $K:=\sum_{i=1}^4 k_{i4}x_ix_4$, we get
$$(\psi_1^{(2)}-\frac{1}{4}K)x_1^2+(\psi_2^{(2)}-\frac{1}{4}K)x_2^+(\psi_3^{(2)}-\frac{1}{4}K)x_3^2=0.$$
Then, by Lemma \ref{polynomial-ch-s21},
\begin{equation*}
\psi_1^{(2)}=\frac{1}{4}K+\alpha x_2^2 +\beta x_3^2, \psi_2^{(2)}=\frac{1}{4}K-\alpha x_1^2 +\gamma x_3^2, \psi_3^{(2)}=\frac{1}{4}K-\beta x_1^2 -\gamma x_2^2
\end{equation*} for some $\alpha,\beta,\gamma \in \RR$.
So we have found all the homogeneous parts of $\psi_1$, $\psi_2$ and $\psi_3$. Now, let us find $\psi_4$.

Comparing the degree 2 terms in \eqref{kolmogorov3-s21}, we get
$$\psi_4^{(0)}=\frac{b^4 -1}{2}k_{44} \text{ and } k_{i4}=0 \text{ for } i=1,2,3.$$
Hence $K=k_{44}x_4^2$. Comparing degree 3 terms in \eqref{kolmogorov3-s21}, we get $\psi_4^{(1)}x_4^2=0$ which follows that $\psi_4^{(1)}=0$. Now comparing degree 4 terms in \eqref{kolmogorov3-s21}, we get
\begin{dmath*}
    -4b^2(\sum_{i=1}^3 \psi_i^{(2)}x_i^2)+2\psi_4^{(2)}x_4^2=k_{44}x_4^2(-2b^2(\sum_{i=1}^3 x_i^2)+x_4^2).
\end{dmath*}
Then using \eqref{homogeneous-2-kolmogorov3-s21} in the above equation, we obtain
\begin{dmath*}
    \psi_4^{(2)}x_4^2=\frac{1}{2}k_{44}x_4^2(-b^2(\sum_{i=1}^3 x_i^2)+x_4^2).
\end{dmath*}
Hence $\psi_4^{(2)}=\frac{1}{2}k_{44}(-b^2(x_1^2+x_2^2+x_3^2)+x_4^2)$. So we have found all homogeneous components of $\psi_4$. Renaming $k_{44}$ as $c$, we get our result.

If $P_1,...,P_4$ are given by \eqref{kolmogorov-form-s21}, then they satisfy \eqref{vectorfield-s21}. Thus, the converse part is true.
\end{proof}
\begin{theorem}\label{thm:noHam_kolm}
    There is no Hamiltonian of a cubic Kolmogorov vector field on $S^2\times S^1$. However, every cubic Kolmogorov vector field on $S^2\times S^1$ has a rational first integral.
\end{theorem}
\begin{proof}
Any cubic Kolmogorov vector field on $S^2\times S^1$ can be described by \eqref{kolmogorov-form-s21}. Suppose $H$ is a Hamiltonian of the vector field \eqref{kolmogorov-form-s21}. Then $$H=x_1x_2(\sum\limits_{i+j=0}^2 a_{ij}x_1^i x_2^j)+x_3x_4(\sum\limits_{i+j=0}^2 b_{ij}x_3^i x_4^j)+c_0x_1x_2x_3x_4$$ by \cite[Theorem 1]{LlXi17}. Since $\frac{\partial H}{\partial x_3}=P_4$, we get 
$$x_4(\sum\limits_{i+j=0}^2 b_{ij}x_3^i x_4^j)+x_3x_4(a_{10}+a_{11}x_4+2a_{20}x_3)+c_0x_1x_2x_4=\frac{1}{2}cx_4(-b^2(x_1^2+x_2^2+x_3^2)+x_4^2+b^4-1).$$
Thus, $c=c_0=0$ and $b_{ij}=0$ for $0\leq i+j\leq 2$. Hence, $H=x_1x_2(\sum\limits_{i+j=0}^2 a_{ij}x_1^i x_2^j)$. Again, $-\frac{\partial H}{\partial x_2}=P_1$ implies that
$$-x_1(\sum\limits_{i+j=0}^2 a_{ij}x_1^i x_2^j)-x_1x_2(a_{01}+a_{11}x_1+2a_{02}x_2)=x_1(\alpha x_2^2+\beta x_3^2).$$ Then, $a_{00}=a_{10}=a_{01}=a_{11}=a_{20}=\beta=0$ and $-3a_{02}=\alpha$. Hence, $H=a_{02}x_1x_2^3$. Now, $\frac{\partial H}{\partial x_1}=P_2$
 implies that $$a_{02}x_2^3=x_2(-\alpha x_1^2+\gamma x_3^2).$$ Thus, $a_{02}=0$ and this gives $H=0$, which is absurd. Hence, a cubic Kolmogorov vector field on $S^2\times S^1$ cannot have any Hamiltonian.

Let $\chi=(P_1, ..., P_4)$ be a vector field of the form \eqref{kolmogorov-form-s21}. Then the hypersurfaces given by $x_1=0,$ ${x_2=0,x_3=0,}$ and $x_1^2+x_2^2+x_3^2=0$ are invariant hypersurfaces $\chi$ with cofactors ${\frac{c}{4}x_4^2+\alpha x_2^2+\beta x_3^2,}$ $ \frac{c}{4}x_4^2-\alpha x_1^2+\gamma x_3^2, \frac{c}{4}x_4^2-\beta x_1^2-\gamma x_2^2,$ and $\frac{c}{2}x_4^2$ respectively. Notice that $$\gamma(\frac{c}{4}x_4^2+\alpha x_2^2+\beta x_3^2)-\beta(\frac{c}{4}x_4^2-\alpha x_1^2+\gamma x_3^2)+\alpha(\frac{c}{4}x_4^2-\beta x_1^2-\gamma x_2^2)-\frac{\gamma-\beta+\alpha}{2}\frac{c}{2}x_4^2=0.$$
Hence by Darboux Integrability Theory (Theorem 5, \cite{LlZh02}), the function
$$f=x_1^{\gamma}x_2^{-\beta}x_3^{\alpha}(x_1^2+x_2^2+x_3^2)^{-\frac{\gamma-\beta+\alpha}{2}}$$ is a rational first integral of $\chi$.
\end{proof}

\begin{lemma}\label{deg-m-n-s21}
Suppose $P_1,P_2,P_3,P_4\in \RR[x_1,...,x_4]$ are non-zero homogeneous polynomials with $\deg(P_i)=m$ for $1\leq i\leq 3$ and $\deg(P_4)=n$. If $m-1\leq n\leq m+3$ then ${G=(x_1^2+x_2^2+x_3^2-b^2)^2 + x_4^2 -1}$ is a first integral of $\chi=(P_1,...,P_4)$.
\end{lemma}
\begin{proof}
    The arguments are similar to the proof of Lemma \ref{deg-m-n-s12}.
\end{proof}

\begin{lemma}\label{lem:P_4-0}
    Let $P_i\in \mathbb{R}[x_1,...,x_4]$ is a homogeneous polynomial for $i=1, ..., 4$ such that  ${G=(x_1^2+x_2^2+x_3^2-b^2)^2 +x_4^2-1}$ is a first integral of $\chi=(P_1,...,P_4)$. Then $P_4=0$.
\end{lemma}
\begin{proof}
    Since $G$ is a first integral of $\chi$, we get the following.
    \begin{dmath*}
        4(x_1^2+x_2^2+x_3^2-b^2)(\sum\limits_{i=1}^3 P_ix_i)+2P_4x_4=0.
    \end{dmath*}
    Hence $x_1^2+x_2^2+x_3^2-b^2$ divides $P_4$. This is possible only when $P_4=0$ since $P_4$ is a homogeneous polynomial.
\end{proof}
\begin{corollary}
    There is no Type-$n$ vector field on $S^2\times S^1$ for $n\in \NN$.
\end{corollary}
\begin{proof}
    If $\chi$ is a Type-$n$ vector field on $S^2\times S^1$ then by Lemma \ref{deg-m-n-s21}, \break $G=(x_1^2+x_2^2+x_3^2-b^2)^2 +x_4^2-1$ is a first integral of $\chi$. Again from Lemma \ref{lem:P_4-0}, if $G$ is a first integral of $\chi$, then $P_4$ must be zero, which contradicts the assumption. Hence there is no Type-$n$ vector field on $S^2\times S^1$.
\end{proof}
\begin{theorem}
\label{pseudo-type-n}
Suppose $\chi=(P_1,...,P_4)$ is a Pseudo Type-$n$ vector field in $\RR^4$. Then $\chi$ is a Pseudo Type-$n$ vector field on $S^2\times S^1$ if and only if there exists $A,B,C\in \RR[x_1,...,x_4]$ such that
\begin{equation}\label{pseudo-type-n-form-s21}
P_1=Ax_2 +Bx_3, P_2=-Ax_1+Cx_3,P_3=-Bx_1-Cx_2,P_4=0.
\end{equation}
\end{theorem}
\begin{proof}
The vector field $\chi$ satisfies \eqref{vectorfield-s21} and also $P_4=0$ since $\chi$ is Pseudo Type-$n$. Hence, we get $P_1x_1 + P_2x_2 +P_3x_3=0$. Now, using Lemma \ref{polynomial-ch-s21}, we get the stated form of a Pseudo Type-$n$ vector field on $S^2\times S^1$.
\par If $P_1,...,P_4$ are given by \eqref{pseudo-type-n-form-s21}, then they satisfy \eqref{vectorfield-s21}. Thus, the converse part is true.
\end{proof}
\begin{theorem}\label{thm:psud_indp}
     There is no Hamiltonian of a Pseudo Type-$n$ vector field on $S^2\times S^1$. However, each Pseudo Type-$n$ vector field on $S^2\times S^1$ has atleast two independent first integrals.
\end{theorem}
\begin{proof}
    Suppose $\chi$ is the Pseudo Type-$n$ vector field \eqref{pseudo-type-n-form-s21}. Assume that $\chi$ is a Hamiltonian vector field with Hamiltonian $H$. $H$ is independent of $x_3$ since $\frac{\partial H}{\partial x_3}=P_4=0$. Hence, $$-\frac{\partial H}{\partial x_2}=P_1=Ax_2+Bx_3 \mbox{ and }  \frac{\partial H}{\partial x_1}=P_2=-Ax_1+Cx_3.$$ Now, we get that $-x_1\frac{\partial H}{\partial x_2}+x_2\frac{\partial H}{\partial x_1}=(Bx_1+Cx_2)x_3$. Then $Bx_1+Cx_2=0$ since $H$ is independent of $x_3$. Hence, $P_3=-(Bx_1+Cx_2)=0$, which contradicts that $\chi$ is a Pseudo Type-$n$ vector field. Hence, $\chi$ is not Hamiltonian.
    
Suppose $f$ is a first integral of the Pseudo Type-$n$ vector field \eqref{pseudo-type-n-form-s21}. Then $\sum\limits_{i=1}^3 P_i \frac{\partial f}{\partial x_i}=0$. This implies
    \begin{equation*}
        (Ax_2+Bx_3)\frac{\partial f}{\partial x_1}+(-Ax_1+Cx_3)\frac{\partial f}{\partial x_2}+(-Bx_1-Cx_2)\frac{\partial f}{\partial x_3}=0.
\end{equation*}
By \cite[Theorem 3(Page 53)]{Sn13}, solutions of the above can be given in terms of the solutions of the equations
$$\frac{dx_1}{Ax_2+Bx_3}=\frac{dx_2}{-Ax_1+Cx_3}=\frac{dx_3}{-Bx_1-Cx_2}=\frac{dx_4}{0}=\frac{df}{0}=\frac{x_1dx_1+x_2dx_2+x_3dx_3}{0}=\frac{dx_4}{0}.$$
Hence, we obtain two independent first integrals, $x_1^2+x_2^2+x_3^2$ and $x_4$.
\end{proof}
\begin{corollary}
    Any Pseudo Type-$n$ Hamiltonian system on $S^2\times S^1$ is integrable in the Liouville sense.
\end{corollary}

\section{Invariant meridian and parallel hyperplanes}\label{sec:mer_par}
This section introduces meridian and parallel hyperplanes in $\RR^4$.  Then we compute an upper bounds for a degree $n$ vector field on $S^1 \times S^2$ and $S^2 \times S^1$. We also show when these bounds can be achieved. We define,
\begin{enumerate}
    \item a hyperplane of the form $a_1x_1+a_2x_2+a_3x_3=0$ as a meridian hyperplane, and
    \item a hyperplane of the form $x_4=k$ as a parallel hyperplane
\end{enumerate}
where $a_1,a_2,a_3,k\in \RR$.

\begin{theorem}
    Suppose $\chi$ is a vector field on $S^1\times S^2$.
\begin{enumerate}[(i)]
    \item If $\chi$ is quadratic, then the number of invariant hyperplanes of the form ${a_1x_1+a_2x_2=0}$ is either 0, 1, or infinite, and the number of invariant hyperplanes of the form $a_3x_3+a_4x_4=0$ is either 0 or 1.
    \item Let $\chi$ be a cubic Kolmogorov vector field such that it has finitely many invariant hyperplanes of the form $a_1x_1+a_2x_2=0$ and $a_3x_3+a_4x_4=0$. Then only such invariant hyperplanes are $x_1=0,x_2=0,x_3=0$, and $x_4=0$.
    \item Let $\chi$ be the Type-$n$ vector field \eqref{type-n-form-s12}. Hyperplane $a_1x_1+a_2x_2=0$ is invariant if and only if $a_1x_1+a_2x_2$ divides $A$. Also, the hyperplane $a_3x_3+a_4x_4=0$ is invariant if and only if $a_3x_3+a_4x_4$ divides $B$. Hence the sharp bound for the number of such invariant hyperplanes of $\chi$ is $n-1$.
\end{enumerate}  
\end{theorem}
\begin{proof}
\begin{enumerate}[(i)]
    \item Suppose $\chi$ is of the form \eqref{quadratic-ch-s12}. Then 
   $$\mathcal{E}_{\langle x_1,x_2\rangle}(\chi)=-f(x_1^2+x_2^2) \text{ and }$$
   $$\mathcal{E}_{\langle x_3,x_4\rangle}(\chi)= \frac{1}{2}(-a^2(x_1^2+x_2^2)+x_3^2+x_4^2+a^4-1)(k_4x_3-k_3x_4)-g(x_3^2+x_4^2).$$
From Proposition \ref{extactic-polynomial}, if $a_1x_1+a_2x_2=0$ is invariant under $\chi$ then $a_1x_1+a_2x_2$ must divide $\mathcal{E}_{\langle x_1,x_2\rangle}(\chi)$, which implies that $a_1x_1+a_2x_2$ divides $f$. Hence if $f$ has no divisor of the form $a_1x_1+a_2x_2$, then $\chi$ has no such invariant hyperplane. If $f$ is a non-zero polynomial having a divisor of the form $a_1x_1+a_2x_2$, then $f$ has exactly one such invariant hyperplane, namely $f=0$. If $f$ is zero polynomial, then one can check that any hyperplane of the form $a_1x_1+a_2x_2=0$ is invariant under $\chi$.

 Suppose $a_3x_3+a_4x_4=0$ is invariant under $\chi$. Then by Proposition \ref{extactic-polynomial}, the polynomial $a_3x_3+a_4x_4$ must divide $\mathcal{E}_{\langle x_3,x_4\rangle}(\chi)$, say $\mathcal{E}_{\langle x_3,x_4\rangle}(\chi)=K^{'}(a_3x_3+a_4x_4)$. Suppose $c$ is the constant term of $K^{'}$. So comparing degree 1 terms in the previous equation, we get $\frac{a^4-1}{2}k_4=ca_3$ and $-\frac{a^4-1}{2}k_3=ca_4$. If $c\neq 0$ then $k_4x_3-k_3x_4=0$ is invariant since $a_3x_3+a_4x_4=\frac{a^4-1}{2c}(k_4x_3-k_3x_4)$. But $\chi(k_4x_3-k_3x_4)=g(k_3x_3+k_4x_4)$ and $k_4x_3-k_3x_4$ doesnot divide $k_3x_3+k_4x_4$, hence $k_4x_3-k_3x_4$ is a divisor of $g$. So $a_3x_3+a_4x_4$ is a divisor of $g$.

If $c=0$ then $k_3=k_4=0$, which implies that $a_3x_3+a_4x_4$ divides $g$ since ${a_3x_3+a_4x_4}$ divides $\mathcal{E}_{\langle x_3,x_4\rangle}(\chi)$. Hence, the number of invariant hyperplanes of the form ${a_3x_3+a_4x_4=0}$ of $\chi$  is either 0 or 1, depending on the divisor of the linear polynomial $g$.
\item This follows immediately from Proposition \ref{extactic-polynomial}.
\item We can check that if $a_1x_1+a_2x_2$ is a factor of $A$ then $a_1x_1+a_2x_2=0$ is invariant hyperplane of $\chi$. Similarly, if $a_3x_3+a_4x_4$ is a factor of $B$ then $a_3x_3+a_4x_4=0$ is invariant hyperplane of $\chi$.
\par Converse parts follow immediately from Proposition \ref{extactic-polynomial}.
\end{enumerate}
\end{proof}

One can see that for a degree $n$ vector field $\chi$, the extactic polynomial $\mathcal{E}_{\langle x_1,x_2,x_3\rangle}(\chi)$ has degree at most $3n$. Hence the number of invariant meridian hyperplanes of $\chi$ is at most $3n$. But in some special cases, we could reduce the bound.

\begin{theorem}
Suppose $\chi$ is a Type-$n$ vector field on $S^1\times S^2$. Assume that $\chi$ has finitely many invariant meridian hyperplanes, taking into account their multiplicities. Then the maximum number of invariant meridian hyperplanes is $3n-2$. Also, this bound is sharp.
\end{theorem}
\begin{proof}
The vector field $\chi=(P_1,...,P_4)$ is of the form \eqref{type-n-form-s12}. Now
$$
\mathcal{E}_{\langle x_1, x_2, x_3\rangle}(\chi) =
\begin{vmatrix}
x_1 &x_2 &x_3\\
Ax_2 &-Ax_1 &Bx_4\\
\chi(Ax_2) &\chi(-Ax_1) &\chi(Bx_4)
\end{vmatrix}
$$
Suppose $C_i$ is the $i$-th column of the above determinant. Then performing the column operation $C_1=x_1C_1+x_2C_2$, we get
$$
    \mathcal{E}_{\langle x_1, x_2, x_3\rangle}(\chi) =
\begin{vmatrix}
x_1^2+x_2^2 &x_2 &x_3\\
0 &-Ax_1 &Bx_4\\
x_1\chi(Ax_2)+x_2\chi(-Ax_1) &\chi(-Ax_1) &\chi(Bx_4)
\end{vmatrix}
$$
Note that $x_1\chi(Ax_2)+x_2\chi(-Ax_1)=x_1(\chi(A)x_2+AP_2)-x_2(\chi(A)x_1+AP_1)=-A^2(x_1^2+x_2^2)$.
Hence, $x_1^2+x_2^2$ divides $\mathcal{E}_{\langle x_1, x_2, x_3\rangle}(\chi)$. But $x_1^2+x_2^2$ has no linear factor over $\RR$. So $\mathcal{E}_{\langle x_1,x_2,x_3\rangle}(\chi)$ can have maximum $3n-2$ divisors of the form $a_1x_1+a_2x_2+a_3x_3$ since the extactic polynomial has degree at most $3n$.
  Hence, by Proposition \ref{extactic-polynomial}, $\chi$ can have at most $3n-2$ invariant hyperplanes(counting multiplicity) of the form $a_1x_1+a_2x_2+a_3x_3=0$.
\par To show that the bound can be obtained, consider the following polynomials.
\begin{equation*}
    P_1=x_1^{n-1}x_2x_3, P_2=-x_1^{n}x_3, P_3=cx_1^{n-1}x_3x_4, P_4=-cx_1^{n-1}x_3^2,
\end{equation*}
where $c\in \RR\backslash \{\pm 1\}, n\geq 2$. Then by Theorem \ref{type-n-s12}, $\chi=(P_1,...,P_4)$ is a Type-$n+1$ vector field on $S^1\times S^2$. One can compute $\mathcal{E}_{\langle x_1,x_2,x_3 \rangle}(\chi)=(x_1^2+x_2^2)(1-c^2)x_1^{3n-3}x_3^4$. Hence, $x_1=0$ and $x_3=0$ are the only invariant hyperplanes of $\chi$ with multiplicities $3n-3$ and $4$, respectively.
\end{proof}
\begin{theorem}
    Suppose $\chi=(P_1,P_2,P_3,0)$ is a Pseudo Type-$n$ vector field on $S^2\times S^1$ and $\chi$ has finitely many invariant meridian hyperplanes, taking into account their multiplicities. If $x_1^2+x_2^2+x_3^2$ divides $P_1,P_2,P_3$ then the maximum number of invariant meridian hyperplanes is $3n-6$.
\end{theorem}
\begin{proof}
The vector field $\chi$ is of the form \eqref{pseudo-type-n-form-s21}. Now
$$
\mathcal{E}_{\langle x_1,x_2,x_3 \rangle}(\chi)=\begin{vmatrix}
    x_1 &x_2 &x_3\\
    P_1 &P_2 &P_3\\
    \chi(P_1) &\chi(P_2) &\chi(P_3)
\end{vmatrix}
$$
Suppose $C_i$ is the $i$-th column of the above determinant. Then performing the column operation $C_1=x_1C_1+x_2C_2+x_3C_3$, we get
$$\mathcal{E}_{\langle x_1,x_2,x_3 \rangle}(\chi)=\begin{vmatrix}
     x_1^2+x_2^2+x_3^2 &x_2 &x_3\\
    x_1P_1+x_2P_2+x_3P_3 &x_2P_2 &x_3P_3\\
    x_1\chi(P_1)+x_2\chi(P_2)+x_3\chi(P_3) &x_2\chi(P_2) &x_3\chi(P_3)
\end{vmatrix}$$
We get $x_1P_1+x_2P_2+x_3P_3=0$ from \eqref{pseudo-type-n-form-s21}. Also
\begin{dmath*}
    \begin{split}
        &x_1\chi(P_1)+x_2\chi(P_2)+x_3\chi(P_3)\\
        =&x_1\chi(Ax_2+Bx_3)+x_2\chi(-Ax_1+Cx_3)+x_3\chi(-Bx_1-Cx_2)\\
        =&x_1(\chi(A)x_2+AP_2+\chi(B)x_3+BP_3)+x_2(-\chi(A)x_1-AP_1+\chi(C)x_3+CP_3)\\ & +x_3(-\chi(B)x_1-BP_1-\chi(C)x_2-CP_2)\\
        =&-P_1(Ax_2+Bx_3)-P_2(-Ax_1+Cx_3)-P_3(-Bx_1-Cx_2)\\
        =&-P_1^2-P_2^2-P_3^2.
    \end{split}
\end{dmath*}
Hence $$\mathcal{E}_{\langle x_1,x_2,x_3 \rangle}(\chi)=\begin{vmatrix}
     x_1^2+x_2^2+x_3^2 &x_2 &x_3\\
    0 &x_2P_2 &x_3P_3\\
    -P_1^2-P_2^2-P_3^2 &x_2\chi(P_2) &x_3\chi(P_3)
\end{vmatrix}.$$

Observe that if $\sum\limits_{i=1}^3 x_i^2$ divides $P_j$, say $P_j=L(x_1^2+x_2^2+x_3^2)$ then $$\chi(P_j)=\chi(L(\sum\limits_{i=1}^3 x_i^2))=\chi(L)(\sum\limits_{i=1}^3 x_i^2)+2L(\sum\limits_{i=1}^3 x_iP_i)=\chi(L)(\sum\limits_{i=1}^3 x_i^2).$$ Hence, $x_1^2+x_2^2+x_3^2$ divides $P_1,P_2,P_3$ implies that $(x_1^2+x_2^2+x_3^2)^3$ is a divisor of $\mathcal{E}_{\langle x_1,x_2,x_3\rangle}(\chi)$ since we can take out $x_1^2+x_2^2+x_3^2$ from the first column, second row and third row.

Now since $\mathcal{E}_{\langle x_1,x_2,x_3\rangle}(\chi)$ has degree $3n$ and $x_1^2+x_2^2+x_3^2$ has no divisor of the form $a_1x_1+a_2x_2+a_3x_3$, the result follows.
\end{proof}
\begin{example}
Consider the Pseudo Type-1 vector field $\chi$ of the form \eqref{pseudo-type-n-form-s21} on $S^2\times S^1$. Now,
$\chi(a_1x_1+a_2x_2+a_3x_3)=(-a_2A-a_3B)x_1 + (a_1A-a_3C)x_2 + (a_1B+a_2C)x_3$. If $\chi$ has invariant hyperplane of the form $a_1x_1+a_2x_2+a_3x_3=0$ then there exists a constant $K$ such that $\chi(a_1x_1+a_2x_2+a_3x_3)=K(a_1x_1+a_2x_2+a_3x_3)$, i.e., $$a_1(Ax_2+Bx_3)+a_2(-Ax_1+Cx_3)+a_3(-Bx_1-Cx_2)=K(a_1x_1+a_2x_2+a_3x_3).$$ Hence comparing coefficients both sides, we obtain $-a_2A-a_3B=Ka_1$, ${a_1A-a_3C=Ka_2}$, $a_1B+a_2C =Ka_3$. Now $$a_1(-a_2A-a_3B)+a_2(a_1A-a_3C)+a_3(a_1B+a_2C)=0=K(a_1^2+a_2^2+a_3^2).$$ So $K=0$ since atleast one of $a_1,a_2,a_3$ is non-zero. This implies $a_2=-a_3 \frac{B}{A}$ and ${a_1=a_3 \frac{C}{A}}$ (assuming $A\neq 0$). Hence $\chi$ has exactly one invariant meridian hyperplane, namely  the hyperplane $Cx_1-Bx_2+Ax_3=0$.
\end{example}
A Type-$n$ vector field on $S^1\times S^2$ has no invariant parallel hyperplane. However, every parallel hyperplane is invariant under a Pseudo Type-$n$ vector field on $S^2\times S^1$.

\begin{theorem}\label{parallel-bound-s12}
Let $\chi$ be a degree $n$ polynomial vector field on $S^1\times S^2$. Assume that $\chi$ has finitely many invariant parallel hyperplanes, taking into account their multiplicities. Then the maximum number of invariant parallel hyperplanes is $(n-1)$. Also, this bound can be reached.
\end{theorem}
\begin{proof}
$\mathcal{E}_{\langle 1,x_4\rangle}=P_4$. By Proposition \ref{extactic-polynomial}, $\chi$ can have maximum $n$ invariant parallel hyperplanes. Suppose $\chi$ has $n$ invariant parallel hyperplanes(counting multiplicities). Then $P_4=c \prod\limits_{i=1}^n (x_4-k_i)$($c\neq 0)$ for some $k_i\in \RR$. As $\chi$ is a vector field on $S^1\times S^2$, it satisfies \eqref{vectorfield-s12} i.e.,
\begin{equation}\label{polynomial-s12}
        4(x_1^2+x_2^2-a^2)(P_1x_1 + P_2x_2)+2(P_3x_3 + P_4x_4)=K((x_1^2+x_2^2-a^2)^2+x_3^2+x_4^2-1).
        \end{equation}
    Equation \eqref{polynomial-s12} is satisfied for each $(x_1,...,x_4)\in \RR$, in particular for $x_1=x_2=x_3=0$. So we get
    $$2cx_4\prod_{i=1}^n (x_4-k_i)=K(0,0,0,x_4)(x_4^2+a^4-1).$$
Note that $K(0,0,0,x_4)\neq 0$ as $c\neq 0$. Also, $x_4\prod\limits_{i=1}^n (x_4-k_i)$ is a factor of $K(0,0,0,x_4)$ since $(x_4^2+a^4-1)$ has no real linear factor. But $K(0,0,0,x_4)$ has degree at most $n-1$, which arises a contradiction. Thus $\chi$ can have at most $n-1$ invariant parallel hyperplanes.
    
    To show that the above bound is reached, consider the following polynomials.
\begin{equation*}
    P_1=Ax_2, P_2=-Ax_1, P_3=x_4\prod_{i=1}^{n-1} (x_4-k_i), P_4=-x_3\prod_{i=1}^{n-1} (x_4-k_i);
\end{equation*}
where $A$ is a polynomial of degree $n-1$ and $k_i$ are distinct real numbers. Then $P_i$'s satisfy \eqref{vectorfield-s12} and thus $\chi=(P_1,...,P_4)$ is a vector field on $S^1\times S^2$. Notice that each parallel hyperplane $x_4-k_i=0$ is invariant for $i\in \{1,...,n-1\}$.
\end{proof}
Theorem \ref{quadratic-ch-s21} tells that a quadratic vector field on $S^2\times S^1$ has either 0 or infinitely many (only when $c=0)$ invariant parallel hyperplanes.
\begin{theorem}
    Let $\chi$ be a degree $n\geq 3$ polynomial vector field on $S^2\times S^1$. Assume that $\chi$ has finitely many invariant parallel hyperplanes, taking into account their multiplicities. Then the maximum number of invariant parallel hyperplanes is $(n-1)$. Also, this bound can be reached.
\end{theorem}
\begin{proof}
    The arguments to the first part are similar to the proof of the first part of Theorem \ref{parallel-bound-s12}.
    
Consider the following polynomials to prove the second part.
    $$
P_1=P_2=P_3=\frac{1}{4}x_4(x_1^2+x_2^2+x_3^2-b^2)\prod\limits_{i=1}^{n-3} (x_4-k_i), P_4=\frac{1}{2}(x_1+x_2+x_3)(x_4^2-1)\prod\limits_{i=1}^{n-3} (x_4-k_i)
$$
where $k_i\neq \pm 1$ are distinct real numbers. Then $P_i$'s satisfy \eqref{vectorfield-s21} and thus $\chi=(P_1,...,P_4)$ is a vector field on $S^2\times S^1$. Notice that $x_4-1=0,x_4+1=0$ and $x_4-k_i=0$ are invariant parallel hyperplanes of $\chi$ for $i=1,...,n-3$. This completes the proof.
\end{proof}

\section{An Application}\label{sec:apps}
We recall the algebraic representation of $S^1\times S^2$ and $S^2\times S^1$ from \eqref{eq:s12} and \eqref{eq:s21} respectively. In this section, we study if there are polynomial diffeomorphisms between $S^1\times S^2$ and $S^2\times S^1$. 
Note that the spaces $S^1\times S^2$ and $S^2\times S^1$ are diffeomorphic via a polynomial map with each component linear polynomial when both the spaces are embedded in $\RR^5$. But this is not true when we consider $S^1\times S^2$ and $S^2\times S^1$ as hyperplanes embedded in $\RR^4$.


Let $F \colon \RR^4 \to \RR^4$ be a polynomial map. If $F|_{S^1 \times S^2} \colon S^1 \times S^2 \to S^2 \times S^1$ then $F|_{S^1 \times S^2}$ is called a polynomial map from $S^1 \times S^2$ to $S^2 \times S^1$. If $F|_{S^1 \times S^2}$ is a diffeomorphism, then it is called a polynomial diffeomorphism from $S^1 \times S^2$ to $S^2 \times S^1$. 

\begin{theorem}\label{thm:nopolydiff}
    There does not exist any polynomial diffeomorphism from $S^1\times S^2$ onto $S^2\times S^1$ with each component as a linear polynomial.
\end{theorem}
\begin{proof}
    Let $F:S^1\times S^2 \to S^2\times S^1$ be a polynomial diffeomorphism where $F=(F_1,...,F_4)$ with each $F_i$ linear. Any $F_i$ cannot be constant. Otherwise, $F$ will not be surjective.

Consider a Type-1 vector field $\chi=(P_1, ..., P_4)$ on $S^1\times S^2$.  Thus, by Corollary \ref{cor:linear-s12},  $$P_1=Ax_2, P_2=-Ax_1, P_3=Bx_4, P_4=-Bx_3$$ where $A$ and $B$ are non-zero real numbers. Next, we compute the push-forward vector field $F_* \chi$ on $S^2\times S^1$.
We consider the Jacobian matrix of $F$.
    $$
    DF(x)=\begin{pmatrix}
        \frac{\partial F_1}{\partial x_1}(x)& .&.&\frac{\partial F_1}{\partial x_4}(x)\\
        .&.&.&.\\
        .&.&.&.\\
        \frac{\partial F_4}{\partial x_1}(x)& .&.&\frac{\partial F_4}{\partial x_4}(x)
    \end{pmatrix}.
    $$
    
Observe that $D:=DF(x)$ is a constant matrix, and no row of $D$ is zero since $F_i$ is non-constant for $i=1, ..., 4$. 

Let $D= DF(x) =(d_{ij})_{4 \times 4}$ and $ {\bf y} \in S^2\times S^1$ such that $F^{-1}({\bf y})=(x_1,...,x_4)$. Note that $F^{-1} =(G_1, ..., G_4)$ with each $G_i$ is non-constant linear. Therefore $$F_* \chi({\bf y})=D (Ax_2, -Ax_1, Bx_4, -Bx_3)^t=D (AG_2({\bf y}), -A G_1{(\bf y)}, BG_4({\bf y}), -B G_3({\bf y}))^t.$$
Hence, it is a linear vector field on $S^2 \times S^1$. Thus, by Corollary \ref{linear-s21}, the fourth component of this vector field is zero. In particular, this is true for $A=0$, $B=1$, and $A=1, B=0$. So, $$d_{43}G_4({\bf y}) - d_{44} G_3({\bf y})=0, \quad \mbox{and} \quad d_{41}G_2({\bf y}) - d_{42} G_1({\bf y})=0$$ for any  $ {\bf y} \in S^2\times S^1$. Since $(d_{41}, d_{42}, d_{43}, d_{44})$ is non-zero and each $G_i$ is non-constant, we have the following cases. 

Suppose, if $d_{41} \neq 0$, then $d_{42} \neq 0$ and $G_2({\bf y}) = \frac{d_{42}}{d_{41}} G_1({\bf y})$. Since $F$ is a diffeomorphism, there exist ${\bf y}_1, {\bf y}_2 \in S^2 \times S^1$ such that $$(1, a, 0,0) = (G_1({\bf y}_1), G_2({\bf y}_1), G_3({\bf y}_1), G_4({\bf y}_1)), ~ (a, 1, 0,0) = (G_1({\bf y}_2), G_2({\bf y}_2), G_3({\bf y}_2), G_4({\bf y}_2)).$$ Thus $$1 = G_2({\bf y}_2)= \frac{d_{42}}{d_{41}} a = \frac{G_2({\bf y}_1)}{G_1({\bf y}_1)} a = a \times a =a^2.$$ This contradicts $a > 1$ which is needed to show that the set in \eqref{eq:s12} is manifold. 

If $d_{41}= 0$, then $d_{42}=0$, otherwise $G_1=0$. So assume, $d_{43}\neq 0$. Then $d_{44} \neq 0$ and $G_4({\bf y}) = \frac{d_{43}}{d_{44}} G_3({\bf y})$. Therefore, $G_3({\bf y}) =0$ for some ${\bf y}$ if and only if  $G_4({\bf y}) = 0$ for those ${\bf y}$. However, $(a, 2^{-\frac{1}{4}}, 2^{-\frac{1}{2}}, 0) \in S^1 \times S^2$. Thus there exists ${\bf y} \in S^2 \times S^1$ such that  $$(a, 2^{-\frac{1}{4}}, 2^{-\frac{1}{2}}, 0) = (G_1({\bf y}), G_2({\bf y}), G_3({\bf y}), G_4({\bf y})).$$ Hence, we arrived at a contradiction. Therefore, the claim is proved. 
\end{proof}

\begin{remark}
The above result justifies why a separate discussion is needed to study polynomial vector fields on $S^1\times S^2$ and $S^2\times S^1$.
\end{remark}

\noindent {\bf Acknowledgment.} 
The authors thank Joji Benny for the helpful discussion. 
The first author is supported by Prime Minister's Research Fellowship, Government of India. The second author thanks ICSR of IIT Madras for SEED grant from April 2021 to March 2024.
\bibliographystyle{abbrv}
\bibliography{biblio.bib}

\begin{thebibliography}{10}

\bibitem{Ar19}
G.~Arutyunov.
\newblock {\em Elements of Classical and Quantum Integrable Systems}.
\newblock UNITEXT for Physics. Springer International Publishing, 2019.

\bibitem{BeSa22}
J.~Benny and S.~Sarkar.
\newblock Certain invariant algebraic sets in the product of spheres.
\newblock {\em arXiv, https://doi.org/10.48550/arXiv.2205.08825}, 2022.

\bibitem{BoLlVa13}
Y.~Bolaños, J.~Llibre, and C.~Valls.
\newblock On the number of invariant conics for the polynomial vector fields
  defined on quadrics.
\newblock {\em Bulletin des Sciences Mathématiques}, 137(6):746--774, 2013.

\bibitem{ChHsWu98}
C.-W. Chi, S.-B. Hsu, and L.-I. Wu.
\newblock On the asymmetric may-leonard model of three competing species.
\newblock {\em SIAM Journal on Applied Mathematics}, 58(1):211--226, 1998.

\bibitem{ChLlPaWa19}
C.~Christopher, J.~Llibre, C.~Pantazi, and S.~Walcher.
\newblock On planar polynomial vector fields with elementary first integrals.
\newblock {\em Journal of Differential Equations}, 267(8):4572--4588, 2019.

\bibitem{ChLlPe07}
C.~Christopher, J.~Llibre, and J.~V. Pereira.
\newblock Multiplicity of invariant algebraic curves in polynomial vector
  fields.
\newblock {\em Pacific Journal of Mathematics}, 229(1):63--117, 2007.

\bibitem{DuLlAr06}
F.~Dumortier, J.~Llibre, and J.~C. Art{\'e}s.
\newblock {\em Qualitative theory of planar differential systems}, volume~2.
\newblock Springer, 2006.

\bibitem{HoSi98}
J.~Hofbauer and K.~Sigmund.
\newblock {\em Evolutionary Games and Population Dynamics}.
\newblock Cambridge University Press, 1998.

\bibitem{LlMe07}
J.~Llibre and J.~C. Medrado.
\newblock On the invariant hyperplanes for $d$-dimensional polynomial vector
  fields.
\newblock {\em Journal of Physics A: Mathematical and Theoretical},
  40(29):8385, 2007.

\bibitem{LlMe11}
J.~Llibre and J.~C. Medrado.
\newblock Limit cycles, invariant meridians and parallels for polynomial vector
  fields on the torus.
\newblock {\em Bulletin des sciences mathematiques}, 135(1):1--9, 2011.

\bibitem{LlMu21}
J.~Llibre and A.~C. Murza.
\newblock Polynomial vector fields on the clifford torus.
\newblock {\em International Journal of Bifurcation and Chaos}, 31(04):2150057,
  2021.

\bibitem{LlPe06}
J.~Llibre and C.~Pessoa.
\newblock Homogeneous polynomial vector fields of degree 2 on the
  2--dimensional sphere.
\newblock {\em Extracta Mathematicae}, 21(2):167--190, 2006.

\bibitem{LlRaRa20}
J.~Llibre, R.~Ram{\'\i}rez, and V.~Ram{\'\i}rez.
\newblock Integrability of a class of $n$-dimensional lotka--volterra and
  kolmogorov systems.
\newblock {\em Journal of differential equations}, 269(3):2503--2531, 2020.

\bibitem{LlRe13}
J.~Llibre and S.~Rebollo-Perdomo.
\newblock Invariant parallels, invariant meridians and limit cycles of
  polynomial vector fields on some 2-dimensional algebraic tori in
  $\mathbb{R}^3$.
\newblock {\em Journal of dynamics and differential equations}, 25:777--793,
  2013.

\bibitem{LlXi17}
J.~Llibre and D.~Xiao.
\newblock Dynamics, integrability and topology for some classes of kolmogorov
  hamiltonian systems in $\mathbb{R}_+^4$.
\newblock {\em Journal of Differential Equations}, 262(3):2231--2253, 2017.

\bibitem{LlZh02}
J.~Llibre and X.~Zhang.
\newblock Darboux integrability of real polynomial vector fields on regular
  algebraic hypersurfaces.
\newblock {\em Rendiconti del circolo matematico di Palermo}, 51(1):109--126,
  2002.

\bibitem{LlZh09}
J.~Llibre and X.~Zhang.
\newblock Darboux theory of integrability in $\mathbb{C}^n$ taking into account
  the multiplicity.
\newblock {\em Journal of Differential Equations}, 246(2):541--551, 2009.

\bibitem{LlZh11}
J.~Llibre and Y.~Zhao.
\newblock On the polynomial vector fields on $\mathbb{S}^2$.
\newblock {\em Proceedings of the Royal Society of Edinburgh Section A},
  141A:055–1069, 2011.

\bibitem{Lo20}
A.~J. Lotka.
\newblock Analytical note on certain rhythmic relations in organic systems.
\newblock {\em Proceedings of the National Academy of Sciences}, 6(7):410--415,
  1920.

\bibitem{Sn13}
I.~Sneddon.
\newblock {\em Elements of Partial Differential Equations}.
\newblock Dover Books on Mathematics. Dover Publications, 2013.

\end{thebibliography}

\end{document}